\documentclass[a4paper,oneside]{article}

\title{Notes on generalised spin structures}
\author{Andrew D.K. Beckett}
\date{July 2026}

\usepackage[utf8]{inputenc} 
\usepackage[hmargin=1in]{geometry}
\usepackage[british]{babel} 

\usepackage[
	backend = biber, 
	style = numeric-comp, 
	giveninits = true,
	sorting = none, 
	maxbibnames = 5, 
	url = false, 
	backref = true 
	]{biblatex}
\AtEveryBibitem{\ifentrytype{article}{\clearfield{issn}}{}}
\AtEveryBibitem{\ifentrytype{article}{\clearfield{editor}}{}}

\usepackage{csquotes}  
\usepackage{mathtools} 
\usepackage{amsthm,amscd,amssymb} 
\usepackage{stmaryrd} 
\usepackage{physics} 
\usepackage{tensor} 
\usepackage[makeroom]{cancel} 
\usepackage{faktor} 
\usepackage{mathrsfs} 
\usepackage{bbold} 
\usepackage{tikz} 
\usetikzlibrary{cd}
\usepackage{enumitem} 
\setenumerate[1]{label={(\arabic*)}}
\usepackage[unicode]{hyperref}
\usepackage{xurl} 
\hypersetup{%
  pdftitle   = {gen-spin},
  pdfkeywords = {generalised spin, spin-G, spinG, symmetry algebra, covariant Lie derivative, Cartan, Cartan calculus, connection, isometry, homogeneous},
  pdfauthor  = {Andrew D.K. Beckett}
  pdfcreator = {\LaTeX\ with package \flqq hyperref\frqq},
  colorlinks = true,
  linkcolor = blue,
  menucolor = blue,
  urlcolor = [rgb]{0.7,0,0},
  citecolor = [rgb]{0,0.7,0},
  breaklinks = true
}

\theoremstyle{plain}
\newtheorem{lemma}{Lemma}
\newtheorem{proposition}[lemma]{Proposition}

\newtheorem{definition}[lemma]{Definition}

\theoremstyle{definition}

\newcommand{\fg}{\mathfrak{g}}
\newcommand{\fh}{\mathfrak{h}}
\newcommand{\fk}{\mathfrak{k}}

\newcommand{\fso}{\mathfrak{so}}
\newcommand{\fiso}{\mathfrak{iso}}

\newcommand{\fX}{\mathfrak{X}}
\newcommand{\fG}{\mathfrak{G}}

\newcommand{\fsymm}{\mathfrak{symm}}
\newcommand{\RR}{\mathbb{R}}

\newcommand{\ZZ}{\mathbb{Z}}

\newcommand{\eA}{\mathscr{A}}

\newcommand{\eL}{\mathscr{L}}

\newcommand{\Wbundle}{\underline{W}}

\newcommand{\Qbar}{\overline{Q}}
\newcommand{\Gbar}{\overline{G}}
\newcommand{\Hbar}{\overline{H}}
\newcommand{\nablahat}{\widehat{\nabla}}
\newcommand{\eLhat}{\widehat{\eL}}
\newcommand{\dhat}{\widehat{\dd}}

\newcommand{\Phat}{\widehat{P}}
\newcommand{\Rhat}{\widehat{R}}
\newcommand{\phihat}{\widehat{\phi}}

\newcommand{\Phihat}{\widehat{\Phi}}

\newcommand{\varphihat}{\widehat{\varphi}}
\newcommand{\sigmahat}{\widehat{\sigma}}

\newcommand{\varphitilde}{\widetilde{\varphi}}
\newcommand{\Ptilde}{\widetilde{P}}
\newcommand{\pihat}{\widehat{\pi}}

\newcommand{\varpihat}{\widehat{\varpi}}

\newcommand{\Wedge}{\mathchoice{{\textstyle\bigwedge}}{\bigwedge}{\bigwedge}{\bigwedge}}	

\newcommand{\Otimes}{\mathchoice{{\textstyle\bigotimes}}{\bigotimes}{\bigotimes}{\bigotimes}}

\newcommand{\1}{\mathbb{1}}
\newcommand{\ccomm}[2]{\qty[\comm{#1}{#2}]}
\DeclareMathOperator{\Id}{Id}

\DeclareMathOperator{\End}{End}

\DeclareMathOperator{\GL}{GL}
\DeclareMathOperator{\Spin}{Spin}
\DeclareMathOperator{\Pin}{Pin}
\DeclareMathOperator{\Orth}{O}
\DeclareMathOperator{\SO}{SO}

\DeclareMathOperator{\Sp}{Sp}

\DeclareMathOperator{\U}{U}
\DeclareMathOperator{\SU}{SU}
\DeclareMathOperator{\Lie}{Lie}

\DeclareMathOperator{\Ad}{Ad}
\DeclareMathOperator{\ad}{ad}

\DeclareMathOperator{\pr}{pr}

\begin{document}

\maketitle

\begin{abstract}
	\noindent
	We review some definitions and basic notions relating to generalised spin structures and introduce the notion of reducibility. We discuss connections on these structures, define a covariant Lie derivative for associated bundles and develop a covariant Cartan calculus. We introduce an extension of the Lie algebra of Killing vectors, the symmetry algebra, and show that it has a representation on sections of associated bundles. We discuss homogeneous generalised spin structures and provide a characterisation of them in terms of lifts of the isotropy representation.
\end{abstract}

\tableofcontents

\section{Introduction}
\label{sec:introduction}

A \emph{generalised spin structure} or \emph{spin-\(G\) structure} is an analogue of a spin structure in which the structure group \(\SO(s,t)\) of the special orthonormal frame bundle of a pseudo-Riemannian manifold is lifted not to the spin group \(\Spin(s,t)\) but to its extension \(\Spin^G(s,t):=(\Spin(s,t)\times G)/\ZZ_2\) by a Lie group \(G\) with a central subgroup isomorphic to \(\ZZ_2\); the \(\ZZ_2\) in the quotient is diagonally embedded. They are of interest to geometers, topologists and physicists for many reasons, not least because they can overcome the well-known topological obstructions to the existence of true spin structures, allowing one to define spinor bundles, Dirac operators and other constructions well known in spin geometry, at the cost of the spinors being ``twisted'' by or ``charged'' with respect to the auxiliary group \(G\).

The first example in the literature of such a generalisation was the spin-\(c\) structure, introduced in 1964 by Atiyah, Bott and Shapiro \cite{Atiyah1964} (see also \cite{Atiyah1973,Whiston1975}\cite[App.D]{Lawson1989}\cite[\S2.4]{Friedrich2000}), which has \(G=\U(1)\); these are well known for their relationship to complex geometry and the celebrated Seiberg-Witten theory.
The spin-\(c\) group arises naturally in the study of complex spinors because it is the product of \(\U(1)\) (the circle of unit scalars) with the spin group inside the complex Clifford algebra.
In positive-definite signature, orientable Riemannian manifolds of dim \(\leq 3\) are spin, with the complex manifold \(\mathbb{CP}^2\) being a well-known example of a 4-manifold which is \textit{not} spin but \textit{is} spin-\(c\) -- in fact, all orientable almost complex manifolds are, as are all orientable 4-manifolds \cite[Prop.5.7.4,Rem.5.7.5]{Gompf1999} (see \cite{Whitney1941,Hirzebruch1958} for the original result in the compact case, and also \cite{LeBrun2021} for an alternative proof).
On the other hand, in dimension 5, the \emph{Wu manifold} \(\SU(2)/\SO(3)\) is a neither spin nor spin-\(c\) \cite[p.50]{Friedrich2000}.

Interest in more general structures was set off in the late 1970s and early 1980s by Hawking and Pope \cite{Hawking1978} who, motivated by problems in the construction of a theory of quantum gravity, proposed using them to define spinor fields on arbitrary 4-manifolds and elucidated their physical significance; this idea was taken up by Back, Freund and Forger who provided an elementary argument that a spin-\(\SU(2)\) structure can be defined on \emph{any} connected, oriented Riemannian 4-manifold \cite{Back1978}.
Generalised spin structures were then put on a firmer mathematical footing by Forger and Hess \cite{Forger1979} and also by Avis and Isham \cite{Avis1980}, with the latter providing in particular a classification for spin-\(\SU(2)\) structures on Riemannian 4-manifolds in terms of characteristic classes. More recent work on physics applications includes \cite{Janssens2018,Lazaroiu2019}.

The case \(G=\SU(2)\cong\Sp(1)\) is the quaternionic analogue of spin-\(c\) in the sense that any almost quaternionic manifold is spin-\(h\); as such, it came to be known as \emph{spin-\(h\)} (or \emph{spin-\(q\)}) when it was taken up by geometers, in particular by Nagase and Bär in the 1990s \cite{Nagase1995,Bar1999} and by others more recently. It has been shown that such structures always exist for orientable Riemannian manifolds in dimension \(\leq 5\), and in the compact case in dimension \(\leq 7\), but there are infinitely many homotopy types of oriented 8-manifolds which do not admit such structures \cite{Lawson2023,Albanese2021}. A treatment for indefinite signature is found in \cite{Lazaroiu2019_1}. A natural generalisation which has also received significant recent interest from geometers is the case \(G=\Spin(k)\) for \(k\in\mathbb{N}\) (giving spin, spin-\(c\) and spin-\(h\) for \(k=1,2,3\) respectively), known as \emph{spinorially twisted spin} or \emph{(generalised) spin-\(k\)} structures \cite{Espinosa2016,Herrera2019,Herrera2018,Albanese2021,Artacho2023}. Two major results here are that every spin-\(k\) manifold admits a codimension \(k\) embedding into a spin manifold, and every manifold is spin-\(k\) for some \(k\), but there is no \(k\) for which every manifold is spin-\(k\) \cite{Albanese2021}.

While any spin-\(G\) structure is \emph{sufficient} for the existence of associated bundles of spinors, the question of the \emph{necessary} structure for such a bundle, or more generally a bundle of real Clifford modules, has been resolved by examining the \emph{real Lipschitz structures} in each signature \((s,t)\) \cite{Lazaroiu2019_1}. These can be described in a number of equivalent ways \cite[Thm.8.2]{Lazaroiu2019_1} but the most relevant for us is the \emph{canonical spinor structure} which is, depending on \(s-t\mod 8\), a spin structure, a pin structure, a spin-\(q\) structure, a pin-\(h\) structure or a \(\Spin^0_{\pm}\) structure; pin-\(h\) is a straightforward generalisation of spin-\(h\) for non-orientable manifolds (\emph{pin-\(k\)} structures for \(k\in\mathbb{N}\) are also discussed in \cite{Albanese2021}) and last type is a spin-\(G\) structure for \(G=\Pin(2,0)\) or \(G=\Pin(0,2)\).

\paragraph{A note on this note}
This note began life as part of the background to another work but has been modified, expanded and released separately, both because it began to exceed its original scope and because it may be of more general interest. Nonetheless, the content is somewhat biased in favour of the author's idiosyncratic applications, so discussion of many important topics is lacking. These include: questions of existence and uniqueness via classifying spaces and characteristic classes; Dirac operators and related topics such as index theorems and Seiberg-Witten theory; detailed examples and applications; and major results both classical and modern. This note may be expanded later into a more complete review of the subject (and as such section and equation numbers etc. are very likely to change!), but for now, the interested reader is encouraged to make good use of the bibliography.

Much of the material which \textit{is} included here is more or less a review of basic notions underlying the existing work cited above, but the following is, however elementary, to the best of the author's knowledge essentially new: the notion of reducibility and related results (\S\ref{sec:reducibility}); the covariant Lie derivative for generalised spin structures (\S\ref{sec:cov-lie-der}, although precursors are acknowledged \cite{deMedeiros2018,Kosmann1971}; see also \cite{Giotopoulos2025} for a review of covariant Lie derivatives in greater generality and \cite{Dalak2026} for covariant Lie derivatives on (s)pinors in generalised geometry); the symmetry algebra (\S\ref{sec:symm-alg-spin-g}); and the classification of homogeneous generalised spin structures in terms of conjugacy classes of lifts of the isotropy representation (\S\ref{sec:homog-gen-spin-class}). The author welcomes correspondence on the originality of this material as well as any other commentary, suggestions or criticism.

\paragraph{Organisation}
This work is organised as follows: in Section~\ref{sec:background} we provide a review of some elementary definitions and results relating to generalised spin structures as well as connections on them and their associated bundles; in Section~\ref{sec:cov-lie-der-cartan} we introduce a \emph{covariant Lie derivative} on the associated bundles determined by a connection and derive some of its properties; in Section~\ref{sec:symm-alg-spin-g} we use the covariant Lie derivative to define an extensions of the Lie algebra of Killing vectors which acts on sections of associated bundles; and finally in Section~\ref{sec:homogeneous-gen-spin} we discuss generalised spin structures on homogeneous spaces.

\section{Basic notions}
\label{sec:background}

\subsection{The spin-\(G\) group}
\label{sec:spin-G-group}

Let \((V,\eta)\) be a (pseudo-)inner product space. We denote the associated (possibly indefinite) orthogonal group by \(\Orth(V,\eta)\) and the special orthogonal group by \(\SO(V,\eta)\). Recall that the spin group \(\Spin(V,\eta)\) has a central subgroup isomorphic to \(\ZZ_2\) which is the kernel of the canonical 2-sheeted covering morphism \(\pi:\Spin(V,\eta)\to\SO(V,\eta)\), giving us the short exact sequence
\begin{equation}\label{eq:spin-cover-so}
\begin{tikzcd}
	1 \ar[r] &\ZZ_2 \ar[r] &\Spin(V) \ar[r,"\pi"] &\SO(V) \ar[r] &1.
\end{tikzcd}
\end{equation}
Now let \(G\) be another Lie group with central \(\ZZ_2\) subgroup and \(\fg\) its Lie algebra. We will denote the identity element of both \(\Spin(V,\eta)\) and \(G\) by \(\1\) and the generator of their distinguished \(\ZZ_2\) subgroups by \(-\1\). We define the \emph{spin-\(G\)} group of \((V,\eta)\) by
\begin{equation}
	\Spin^G(V,\eta) := \faktor{\Spin(V,\eta)\times G}{\{(\pm\1,\pm\1)\}}.
\end{equation}
We typically omit \(\eta\) in the notation where this is not ambiguous, and we also write \(\Spin^G(s,t):=\Spin^G(\RR^{s,t})\) for the spin-\(G\) group of a standard inner product space. Considering an action of \(\ZZ_2\) by multiplication by \(\pm \1\) on the right of \(\Spin(V)\) and on the left of \(G\), we can also view \(\Spin^G(V)\) as the balanced product \(\Spin(V)\times_{\ZZ_2}G\). We write the image of a pair \((a,g)\in\Spin(V)\times G\) in \(\Spin^G(V)\) as \([a,g]\).

By definition, we have a short exact sequence
\begin{equation}\label{eq:spin-G-SES-defining}
\begin{tikzcd}
	1 \ar[r] &\ZZ_2 \ar[r] &\Spin(V)\times G \ar[r,"\nu"] &\Spin^G(V) \ar[r] &1,
\end{tikzcd}
\end{equation}
where \(\nu\) is the natural morphism \(\nu(a,g)=[a,g]\); it is a two-sheeted covering of \(\Spin^G(V)\). 
We have two natural injective Lie group morphisms
\begin{align}
	& \Spin(V) \xhookrightarrow{\quad} \Spin^G(V),	& a&\longmapsto [a,\1],\\
	& G \xhookrightarrow{\quad} \Spin^G(V),			& g&\longmapsto [\1,g],
\end{align}
and two surjective morphisms
\begin{align}
	\pihat&: \Spin^G(V) {\relbar\joinrel\twoheadrightarrow} \SO(V),		& [a,g]&\longmapsto \pi(a),\\
	\pi_G &: \Spin^G(V) {\relbar\joinrel\twoheadrightarrow} \Gbar = \faktor{G}{\ZZ_2},	& [a,g]&\longmapsto \bar{g},
\end{align}
where we note that we have defined \(\Gbar=G/\ZZ_2\) and \(\bar{g}\in \Gbar\) as the coset of \(g\in G\); we also denote by \(p_G(g)=\bar{g}\) the natural map \(p_G:G\to\Gbar\). 
Together with the sequence \eqref{eq:spin-cover-so}, these maps fit into the following diagram with exact rows and columns:
\begin{equation}\label{eq:spin-G-SES-square}
\begin{tikzcd}
	 & 1 \ar[d] & 1\ar[d]	\\
	1 \ar[r] & \ZZ_2 \ar[r]\ar[d] & \Spin(V) \ar[r,"\pi"]\ar[d] & \SO(V) \ar[r]\ar[d,equal] & 1	\\
	1 \ar[r] & G \ar[r]\ar[d,"p_G"] & \Spin^G(V) \ar[r,"\pihat"]\ar[d,"\pi_G"] & \SO(V) \ar[r] & 1	\\
	 & \Gbar \ar[r,equal]\ar[d] & \Gbar \ar[d]	\\
	 & 1 & 1
\end{tikzcd}.
\end{equation}
The map \(\pihat\) will be particularly important in our discussion of spin-\(G\) structures. 

We also have the short exact sequence
\begin{equation}\label{eq:spin-G-SES-SO-GmodZ2}
\begin{tikzcd}
	1 \ar[r] &\ZZ_2 \ar[r] &\Spin^G(V) \ar[r,"\pihat\times\pi_G",pos=0.4] &\SO(V)\times\Gbar \ar[r] &1
\end{tikzcd}
\end{equation}
where the surjection is given by \([a,r]\mapsto(\pi(a),\bar{r})\) and the kernel \(\ZZ_2\) is generated by the element \([\1,-\1]=[-\1,\1]\in\Spin^G(V)\); in particular, \(\Spin^G(V)\) is a two-sheeted cover of \(\SO(V)\times\Gbar\). This sequence fits with \eqref{eq:spin-G-SES-defining} into another diagram with exact rows and columns
\begin{equation}\label{eq:spin-G-SES-square-2}
\begin{tikzcd}
	 		& 1 \ar[d] & 1 \ar[d] & 1 \ar[d]	\\
	1 \ar[r]& \ZZ_2 \ar[r]\ar[d,equal] & \ZZ_2\times\ZZ_2 \ar[r]\ar[d] & \ZZ_2 \ar[r]\ar[d] & 1	\\
	1 \ar[r]& \ZZ_2 \ar[r] & \Spin(V)\times G \ar[r,"\nu"]\ar[d,"\pi\times p_G"] & \Spin^G(V) \ar[r]\ar[d,"\pihat\times\pi_G"] & 1	\\
			& & \SO(V)\times\Gbar \ar[r,equal]\ar[d]	& \SO(V)\times\Gbar \ar[d]\\
	 		& & 1 & 1
\end{tikzcd}
\end{equation}
where the \(\ZZ_2\times\ZZ_2\) in the top row is the subgroup of \(\Spin(V)\times G\) generated by \(\{(\1,-\1),(-\1,\1)\}\), the \(\ZZ_2\) on the left is the subgroup with generator \((-\1,-\1)\) and the \(\ZZ_2\) on the right has the generator \([\1,-\1]=[-\1,\1]\in\Spin^G(V)\). In particular, we have a commutative diagram of group covering maps
\begin{equation}\label{eq:spin-G-double-covers}
\begin{tikzcd}
	\Spin(V)\times G \ar[rr, two heads, "\nu"']\ar[rrrr,bend left=12, two heads,"\pi\times p_G"] 
	&& \Spin^G(V) \ar[rr, two heads, "\pihat\times\pi_G"'] 
	&& \SO(V)\times\Gbar,
\end{tikzcd}
\end{equation}
so the Lie algebras of these three groups are naturally isomorphic:
\begin{equation}
	\mathfrak{spin}^G(V) = \Lie\Spin^G(V)\cong \Lie(\Spin(V)\times G) 
		\cong \Lie\qty(\SO(V)\times \Gbar) 
		\cong \fso(V)\oplus \fg.
\end{equation}
Denoting the adjoint action of any group \(G\) by \(\Ad^G\), we have
\begin{equation}\label{eq:spin-G-adjoint}
	\Ad^{\Spin^G(V)}_{[a,g]}(X,x) 
		= \qty(\Ad^{\Spin(V)}_a X, \Ad^G_g x)
		= \qty(\Ad^{\SO(V)}_{\pi(a)} X, \Ad^{\Gbar}_{\bar{g}} x)
\end{equation}
for all \(a\in \Spin(V)\), \(g\in G\), \(X\in\fso(V)\) and \(x\in \fg\).

\subsection{Spin-\(G\) structures}
\label{sec:spin-G-structures}

Now let \((M,g)\) be an oriented pseudo-Riemannian manifold of signature \((s,t)\) and \(F_{SO}\to M\) its special orthonormal frame bundle, which has the structure of an \(\SO(s,t)\)-principal bundle. Recall that a \emph{spin structure} on \((M,g)\) is a \(\Spin(s,t)\)-principal bundle \(P\to M\) with a \(\Spin(s,t)\)-equivariant bundle map \(\varpi:P\to F_{SO}\) where \(\Spin(s,t)\) acts on \(F_{SO}\) via the natural covering group morphism \(\pi:\Spin(s,t)\to\SO(s,t)\). Now note that \(\Spin^G(s,t)\) also acts on \(F_{SO}\) via the natural map \(\pihat:\Spin^G(s,t)\to\SO(s,t)\).

\begin{definition}[Spin-\(G\) structure]\label{def:spin-G-struc}
	A \emph{spin-\(G\) structure} on an oriented pseudo-Riemannian manifold \((M,g)\) of signature \((s,t)\) is a \(\Spin^G(s,t)\)-principal bundle \(\Phat\to M\) with a bundle map \(\varpihat:\Phat \to F_{SO}\) which is \(\pihat\)-equivariant; that is, \(\varpihat(p\cdot [a,g])=\varpihat(p)\cdot\pihat([a,g])=\varpihat(p)\cdot\pi(a)\) for all \(a\in\Spin(s,t)\) and \(g\in G\).
	
	Two such structures \(\varpihat:\Phat\to M\), \(\varpihat':\Phat'\to M\) are \emph{equivalent} if there exists an isomorphism of \(\Spin^G(s,t)\)-principal bundles \(\Phi:\Phat\to\Phat'\) such that the following diagram commutes:
	\begin{equation}
	\begin{tikzcd}
		& \Phat\ar[rr,"\Phi"]\ar[dr,"\varpihat"']	&	& \Phat' \ar[dl,"\varpihat'"]\\
		&	& F_{SO} &
	\end{tikzcd}.
	\end{equation}
	We will call \((M,g)\) a \emph{spin-\(G\) manifold} if it is equipped with a spin-\(G\) structure.
\end{definition}

Equivalently, a spin-\(G\) structure is a lift of the structure group of \(F_{SO}\) along \(\pihat\). i.e. an isomorphism of \(\SO(s,t)\)-principal bundles \(\widehat\Pi:\Phat\times_{\pihat}\SO(s,t)\xrightarrow{\cong} F_{SO}\); the relationship between the two definitions is given by \(\widehat\Pi([p,A])= \varpihat(p)\cdot A\). On the other hand, such a map \(\widehat\Pi\) defines \(\varpihat\) by defining \(\varpihat(p)=\widehat\Pi([p,\1])\).

Unlike a spin structure, a spin-\(G\) structure is not in general a covering of \(F_{SO}\); instead, \(\varpihat: \Phat \to F_{SO}\) is a \(G\)-principal bundle on \(F_{SO}\), where \(G\) acts on \(\Phat\) via the natural homomorphism \(G\hookrightarrow \Spin^G(s,t)\). On the other hand, \(\Phat\) can be viewed as a cover of a larger bundle as follows. The homomorphism \(\pi_G:\Spin^G(s,t)\to \Gbar\) induces a bundle map
\begin{equation}
	\varpi_G: \Phat \longrightarrow \Qbar := \Phat\times_{\pi_G}\Gbar
\end{equation}
defined by \(\varpi_G(p)=[p,\1]\) which is \(\pi_G\)-equivariant in the sense that \(\varpi(p\cdot [a,g])=\varpi(p)\cdot\pi_G([a,g])=\varpi(p)\cdot\bar{g}\), and we call \(\Qbar\to M\) the \emph{canonical \(\Gbar\)-bundle} -- note that we can consider \(\varpi_G\) as a \(\Spin(s,t)\)-principal bundle structure over \(\Qbar\). The map \(\pihat\times\pi_G:\Spin^G(s,t)\to\SO(s,t)\times \Gbar\) in the short exact sequence \eqref{eq:spin-G-SES-SO-GmodZ2} induces a two-sheeted covering map
\begin{equation}
	\varpihat\times\varpi_G:\Phat \longrightarrow F_{SO}\times_M\Qbar.
\end{equation} 
We can summarise the relationships between these principal bundles, each of which can be considered as an associated bundle to \(\Phat\), in the following commutative diagram of \(\Spin^G(s,t)\)-equivariant bundle maps:
\begin{equation}
\begin{tikzcd}[sloped]
	\Phat \ar[dr,"\ZZ_2"]\ar[drr,"G",bend left=12] 
		\ar[ddr,"{\Spin}"', bend right=12] 
		\ar[ddrr,"{\Spin^G}",controls={+(5,1) and +(1,1)}] & &	\\
	& F_{SO}\times_M \Qbar \ar[r,"\Gbar"]\ar[d,"{\SO}"] 
		\ar[dr,"{\SO\times\Gbar}"]	& F_{SO} \ar[d,"{\SO}"]	\\
	& \Qbar \ar[r,"\Gbar"]	& M	
\end{tikzcd}
\end{equation}
where each arrow represents a principal bundle over the target and is labelled by the structure group (omitting \((s,t)\) for clarity); equivalently, each represents a lift of the structure group along an epimorphism of Lie groups and is labelled with the kernel of that morphism.

\subsection{Reducibility}
\label{sec:reducibility}

Now let \(\varpi:P\to F_{SO}\) be a spin structure and \(Q\to M\) a principal \(G\)-bundle. Then the fibred product \(P\times_M Q\to M\) is a \(\Spin(s,t)\times G\)-principal bundle on \(M\), and the canonical homomorphisms \(\nu:\Spin(s,t)\times G\to\Spin^G(s,t)\) and \(\pihat:\Spin^G(s,t)\to\SO(s,t)\) induce the spin-\(G\) structure
\begin{equation}\label{eq:spin-G-reducible-PQ}
	\varpihat_{PQ}:\Phat_{PQ}:= (P\times_M Q)\times_\nu\Spin^G(s,t) \longrightarrow F_{SO},
\end{equation}
where \(\varpihat_{PQ}([(p,q),[a,g]])=\varpi(p)\cdot\pihat([a,g])=\varpi(p\cdot a)\). We also have a natural map \(\psi_{PQ}:P\times_M Q\to \Phat_{PQ}\) given by \(\psi_{PQ}(p,q)=[(p,q),\1]\) which is a \(\nu\)-equivariant 2-sheeted covering. There are canonical isomorphisms of \(\SO(s,t)\)-bundles 
\begin{equation}
	\Phat_{PQ}\times_{\pihat}\SO(s,t) \xlongrightarrow{\cong} P\times_\pi\SO(s,t) \xlongrightarrow{\cong} F_{SO}
\end{equation}
given by \([[p,q],A]\mapsto[p,A]\mapsto\varpi(p)\cdot A\) and a canonical isomorphism of \(\Gbar\)-bundles
\begin{equation}
	\Qbar = \Phat_{PQ}\times_{\pi_G}\Gbar \xlongrightarrow{\cong} Q\times_{p_G} \Gbar
\end{equation}
given by \([[p,q],\bar{g}]\mapsto [q,\bar{g}]\). In particular, \(Q\) is a lift of the structure group of \(\Qbar\) along \(p_G:G\to\Gbar\); we have a two-sheeted \(G\)-equivariant covering \(\varpi_Q: Q \to \Qbar\) given by \(\varpi_Q(q)=[(p,q),\bar\1]\) for any element \(p\in P\) -- for any other \(p'\in P\), we have \(p'=p\cdot a\) for some \(a\in\Spin(s,t)\) so \([(p',q),\bar\1]=[(p,q)\cdot(a,\1),\1]=[(p,q),\1]\).

Now considering each bundle as an associated bundle to \(P\times_M Q\), we can extend the diagram above to the following diagram of \(\Spin(s,t)\times G\)-equivariant bundle maps
\begin{equation}
\begin{tikzcd}[sloped,column sep=12,row sep=12]
	P\times_M Q \ar[dr,"\ZZ_2"]
		\ar[dddd,"{\Spin}"'] \ar[rrrr,"G"]
		& & & & P \ar[dd,"\ZZ_2",pos=0.35] \ar[dddd,"{\Spin}",bend left=60] \\
	& \Phat_{PQ} \ar[dr,"\ZZ_2"]\ar[drrr,"G", bend left=12]
		\ar[dddr,"{\Spin}"', bend right=12,pos=0.6]
		& & &	\\
	& & F_{SO}\times_M \Qbar \ar[rr,"\Gbar"]\ar[dd,"{\SO}"] \ar[ddrr,"{\SO\times\Gbar}"] 
		\ar[from=uull,"\ZZ_2\times\ZZ_2"',bend right=22,crossing over,pos=0.55]
		& & F_{SO} \ar[dd,"{\SO}"]	\\
	& & & & \\
	Q \ar[rr,"\ZZ_2",pos=0.3] 
		\ar[rrrr,"G",controls={+(1,-1.2) and + (-1,-1.2)},pos=0.4]
		& & \Qbar \ar[rr,"\Gbar"]
		& & M
		\ar[from=uuulll,"{\Spin^G}",controls={+(3.5,1) and +(1.5,2)},crossing over,pos=0.3] 
		\ar[from=uuuullll,"{\Spin\times G}"',controls={+(1.9,-4.8) and +(-1.5,-0.7)},crossing over,pos=0.2]
\end{tikzcd}
\end{equation}
where again the arrows represent principal bundle structures and are labelled by structure groups.

In particular, we have the following sub-diagram of \(\Spin(s,t)\times G\)-equivariant covering maps, which we can also see as being induced by by the diagram of group coverings \eqref{eq:spin-G-double-covers}:
\begin{equation}\label{eq:PQ-double-covers}
\begin{tikzcd}
	P\times_M Q \ar[rr, two heads,"\psi_{PQ}","\ZZ_2"'] \ar[rrrr,bend left=18, two heads,"\varpi\times \varpi_Q","\ZZ_2\times\ZZ_2"'] 
	&&\Phat_{PQ} \ar[rr, two heads, "\varpihat_{PQ}\times\varpi_G","\ZZ_2"'] 
	&&F_{SO}\times_M\Qbar.
\end{tikzcd}
\end{equation}
Of course, not every spin-\(G\) structure arises in this way -- \((M,g)\) need not admit a spin structure nor a lift of the structure group of the canonical \(\Gbar\)-bundle to \(G\) -- hence the following definition and result.

\begin{definition}[Reducible spin-\(G\) structure]\label{def:spin-G-reducible}
	A spin-\(G\) structure \(\varpihat:\Phat\to F_{SO}\) is \emph{reducible} if it admits a lift of the structure group of \(\Phat\to M\) along \(\nu:\Spin(s,t)\times G\to\Spin^G(s,t)\).
\end{definition}

That is, \(\Phat\to F_{SO}\) is reducible if there exists a \(\Spin(s,t)\times G\)-principal bundle \(\Ptilde\to M\) and a  \(\nu\)-equivariant bundle map \(\psi: \Ptilde\to\Phat\), where the product group acts on \(\Phat\) via \(\nu\), or equivalently an isomorphism of \(\Spin^G(s,t)\)-bundles \(\Psi:\Ptilde\times_\nu\Spin^G(s,t)\to\Phat\) where the relationship between the two maps is given by \(\Psi([\tilde{p},[a,g]])=\psi(\tilde{p})\cdot [a,g]\). Thus a spin-\(G\) structure of the form \eqref{eq:spin-G-reducible-PQ} is, by definition, reducible with \(\Ptilde=P\times_M Q\). In fact, the converse is also true.

\begin{proposition}\label{prop:spin-G-reducible}
	A spin-\(G\) structure \(\varpihat:\Phat\to F_{SO}\) is reducible with lift \(\psi:\Ptilde\to \Phat\) if and only if there exists a spin structure \(\varpi:P\to F_{SO}\), a lift \(\varpi_Q:Q\to \Qbar\) of the structure group of the canonical \(\Gbar\)-bundle along \(p_G:G\to\Gbar\), and an equivalence of spin-\(G\) structures \(\Phi:\Phat\to\Phat_{PQ}\). Moreover, this equivalence is covered by an isomorphism of \(\Spin(s,t)\times G\)-principal bundles \(\widetilde{\Phi}:\Ptilde\to P\times_M Q\) so that diagram
	\begin{equation}
	\begin{tikzcd}
		& \Ptilde \ar[rr,"\widetilde{\Phi}"]\ar[d,"\psi"']	&& P\times_M Q \ar[d,"\psi_{PQ}"]	\\
		& \Phat \ar[rr,"\Phi"]\ar[dr,"\varpihat"'] 	&& \Phat_{PQ} \ar[dl,"\varpihat_{PQ}"]	\\
		& & F_{SO}	&	
	\end{tikzcd}
	\end{equation}
	commutes.
\end{proposition}

\begin{proof}
	If there exists an equivalence \(\Phi:\Phat\to\Phat_{PQ}\) then pulling back the natural map \(\psi_{PQ}:P\times_M Q\to\Phat_{PQ}\) along \(\Phi\) gives the required lift of the structure group of \(\Phat\) as well as the rest of the diagram.
	
	On the hand, given a lift \(\psi:\Ptilde\to \Phat\), we define a spin structure and a lift 
	\begin{equation}
	\begin{aligned}
		\varpi: P:=\Ptilde\times_{\pr_{\Spin}}\Spin(s,t) \longrightarrow F_{SO},
		&& \varpi_Q: \Ptilde\times_{\pr_{G}} G \longrightarrow \Qbar = \Phat\times_{\pi_G} \Gbar
	\end{aligned}
	\end{equation}
	by \(\varpi([\tilde{p},a])=\varpihat(\psi(\tilde{p}))\cdot\pi(a)\) and \(\varpi_Q([\tilde{p},g])=[\psi(\tilde{p}),\bar{g}]\), where \(\tilde{p}\in\Ptilde\), \(a\in\Spin(s,t)\) and \(g\in G\);	these maps are well-defined and respectively \(\Spin(s,t)\)- or \(G\)-equivariant by the equivariance properties of \(\varpihat\) and \(\psi\). We define the isomorphism \(\widetilde\Phi:\Ptilde\xrightarrow{\cong} P\times_M Q\) of \(\Spin(s,t)\times G\)-principal bundles by \(\widetilde\Phi(\tilde{p})=([\tilde{p},\1],[\tilde{p},\1])\).
	The composition \(\psi_{PQ}\circ\widetilde\Phi:\Ptilde\to\Phat_{PQ}\) maps \(\tilde{p}\mapsto [([\tilde{p},\1],[\tilde{p},\1]),\1]\) and we note that if \(\tilde{p},\tilde{p}'\in \Ptilde\) such that \(\psi(\tilde{p})=\psi(\tilde{p}')\) then \(\tilde{p}'=\tilde{p}\cdot[\pm\1,\pm\1]\), whence by the equivariance properties of \(\widetilde\Phi\) and \(\psi_{PQ}\),
	\begin{equation}
		\qty(\psi_{PQ}(\widetilde\Phi(\tilde{p}')))
			= \psi_{PQ}\qty(\widetilde\Phi(\tilde{p})\cdot [\pm\1,\pm\1])
			= \psi_{PQ}\qty(\widetilde\Phi(\tilde{p}))\cdot \nu([\pm\1,\pm\1])
			= \psi_{PQ}\qty(\widetilde\Phi(\tilde{p})),
	\end{equation}
	so \(\psi_{PQ}\circ\widetilde\Phi\) factors through \(\psi\); that is, there is a unique isomorphism of \(\Spin^G(s,t)\)-bundles \(\Phi:\Phat\to\Phat_{PQ}\) making the top square of the diagram commute. It remains to show that \(\varpihat_{PQ}\circ\Phi=\varpihat\). 
	For any \(\hat{p}\in\Phat\), let \(\tilde{p}\in\Ptilde\) be a lift, i.e. \(\psi(\tilde{p})=\hat{p}\); then, using only definitions,
	\begin{equation}
		\varpihat_{PQ}\qty(\Phi(\hat{p}))
			= \varpihat_{PQ}\qty(\psi_{PQ}\qty(\widetilde{\Phi}(\tilde{p})))
			= \varpihat_{PQ}([([\tilde{p},\1],[\tilde{p},\1]),\1])
			= \varpi([\tilde{p},\1])
			= \varpihat(\psi(\tilde{p}))
			= \varpihat(\hat{p})
	\end{equation} 
	as required.
\end{proof}

\subsection{Time orientability}
\label{sec:time-orient}

A time orientation is a reduction of the structure group of \(F_{SO}\to M\) to \(\SO_0(s,t)\), the connected component of the identity of \(\SO(s,t)\) (known as the \emph{proper orthochronous Lorentz group} in Lorentzian signature); we denote the corresponding sub-bundle by \(F_{SO_0}\to M\). Let us denote the connected component of the identity of the spin group by \(\Spin_0(s,t)\) and recall that the spin-\(R\) structure map is denoted \(\varpihat:\Phat\to F_{SO}\). Give a time orientation, we can reduce the structure group of the \(\Spin^G(s,t)\)-principal bundle \(\Phat\to M\) to the index-2 subgroup\footnote{
		Note that although \(\Spin_0(s,t)\) is connected, \(\Spin^G_0(s,t)\) may not be since we do not assume that \(G\) is connected.
		}
\(\Spin^G_0(s,t):=\Spin_0(s,t)\times_{\ZZ_2} G\) as follows: we define  \(\Phat_0:=\varpihat^{-1}(F_{SO_0})\), which carries a natural right action by \(\Spin^G_0(s,t)\) obtained by restricting the action by \(\Spin^G(s,t)\) on \(\Phat\); the restriction of the projection map \(\Phat\to M\) to \(P_0\) then gives us a \(\Spin^G_0(s,t)\)-principal bundle \(\Phat_0\to M\). On the other hand, if there exists a reduction \(\Phat_0\to M\) of the structure group of \(\Phat\to M\) to \(\Spin^G_0(s,t)\) then \(F_{SO_0}:=\varpihat(\Phat_0)\) is a time orientation. We have thus shown the following.

\begin{lemma}\label{lemma:time-orientable-spin-G}
	Let \((M,g)\) be an oriented pseudo-Riemannian manifold of indefinite signature with spin-\(G\) structure \(\varpihat:\Phat\to F_{SO}\). Then \((M,g)\) is time-orientable if and only if there exists a reduction \(\Phat_0\to M\) of the structure group of the \(\Spin^G(s,t)\)-principal bundle \(\Phat\to M\) to the subgroup \(\Spin^G_0(s,t)\).
\end{lemma}

Of course, we can also restrict \(\varpihat\) to a \(\Spin^G_0(s,t)\)-equivariant map \(\varpihat_0:\Phat_0\to F_{SO_0}\), where the action on \(F_{SO_0}\) is via \(\pihat_0:=\pihat|_{\Spin^G_0(s,t)}: \Spin^G_0(s,t) \to\SO_0(s,t)\). This is the spin-\(G\) analogue of a \emph{strong spin structure} \cite{Cortes2021,Shahbazi2024_1,Shahbazi2024_2}.

\subsection{Representations of the spin-\(G\) group}
\label{sec:spin-G-reps}

Any representation of \(\Spin^G(V)\) pulls back along \(\nu:\Spin(V)\times G\to\Spin^G(V)\) and thus also pulls back to representations of \(\Spin(V)\) and of \(G\). Conversely, two representations \(\varrho_1: \Spin(V)\to \GL(W)\) and \(\varrho_2:G\to \GL(W)\) on the same space \(W\) assemble into a representation \(\varrho_1\varrho_2:\Spin(V)\times G\to\GL(W)\) given by \((\varrho_1\varrho_2)(a,r)=\varrho_1(a)\circ\varrho_2(r)\) if and only if \(\varrho_1(a)\circ\varrho_2(g)=\varrho_2(g)\circ \varrho_1(a)\) for all \(a\in\Spin(V)\) and \(g\in G\). This factors through a representation \(\varrho:\Spin^G(V)\to \GL(W)\) (such that \(\varrho\circ\nu=\varrho_1\varrho_2\)) if and only if \(\ker\nu=\{(\pm\1,\pm \1)\}\) acts trivially; that is, if 
\begin{equation}\label{eq:spinr-invols-inverse}
	\varrho_1(-\1)\circ\varrho_2(-\1) = \varrho_2(-\1)\circ \varrho_1(-\1) = \Id_W,
\end{equation}
or equivalently \(\varrho_2(-\1)=\varrho_1(-\1)^{-1}\). On the other hand,
\begin{equation}\label{eq:spinr-invol}
	\varrho_i(-\1)^2 = \varrho_i((-\1)^2) = \varrho_i(\1) = \Id_W,
\end{equation}
identically, so by uniqueness of inverses, the condition \eqref{eq:spinr-invols-inverse} is equivalent to
\begin{equation}\label{eq:spinr-invols-agree}
	\varrho_1(-\1)=\varrho_2(-\1).
\end{equation}
This condition is satisfied in particular if \(\varrho_1(-\1)=\varrho_2(-\1)=\pm \Id_W\), so we make the following definition.

\begin{definition}
	A representation \(\varrho:\Spin^G(V)\to \GL(W)\) with pull-backs \(\varrho_1:\Spin(V)\to\GL(W)\) and \(\varrho_2:G\to\GL(W)\) is \emph{even} if \(\varrho_1(-\1)=\varrho_2(-\1)=\Id_W\) and it is \emph{odd} if \(\varrho_1(-\1)=\varrho_2(-\1)=-\Id_W\).
\end{definition}

For example, \(\pihat:\Spin^G(V)\to\SO(V)\) gives us an even representation on \(V\) in which \(G\) acts trivially, and the adjoint representation of \(\Spin^G(V)\) is also even. Since
\begin{align*}
	\varrho_1(-\1)= \Id_W &\iff \text{\(\varrho_1\) factors through a representation} \quad \overline{\varrho}_1:\SO(V) \to \GL(W),\\
	\varrho_2(-\1)= \Id_W &\iff \text{\(\varrho_2\) factors through a representation} \quad \overline{\varrho}_2: \Gbar = \faktor{G}{\ZZ_2} \to \GL(W),
\end{align*}
even representations of \(\Spin^G(V)\) are those which factor through a representation \(\overline\varrho=\overline\varrho_1\overline\varrho_2\) of \(\SO(V)\times \Gbar\) such that \(\overline\rho\circ(\pihat\times\pi_G)=(\overline\rho_1\circ\pihat)(\overline\rho_2\circ\pi_G)=\rho\).

It is often useful for applications to build odd representations out of spinors as follows. Let \(\sigma:\Spin(V)\to\GL(S)\) be an irreducible spinor representation, note that \(\sigma(-\1)=-\Id_S\), and let \(\delta:G\to\GL(\Delta)\) be a finite-dimensional representation of \(G\) in which \(\delta(-\1)=-\Id_\Delta\). Then the representation
\begin{equation}
\begin{aligned}
	&\sigma\otimes\delta:\Spin(V)\times G\to \GL(S\otimes\Delta),
	&& (\sigma\otimes\delta)(a,g)(s\otimes x) = \sigma(a)(s)\otimes\delta(a)(x)	
\end{aligned}
\end{equation}
for \(a\in \Spin(V)\), \(g\in G\), \(s\in S\), \(x\in\Delta\)
 factors through an odd representation of \(\Spin^G(V)\).

An alternative situation is where \(\sigma:\Spin(V)\to\GL(S)\) is a spinor module which is not necessarily irreducible (i.e. it might be a sum of irreducibles)  and there is an action \(\sigma_G:G\to\GL(S)\) by automorphisms of \(\sigma\) such that \(\sigma_G(-\1)=-\Id_S\); in this case, \(\sigma\sigma_G\) factors through a representation \(\sigmahat:\Spin^G(V)\to\GL(S)\).

\begin{lemma}
	Any representation of \(\Spin^G(V)\) is a direct sum of even and odd representations. In particular, irreducible representations of \(\Spin^G(V)\) are either even or odd.
\end{lemma}

\begin{proof}
	Let \(\varrho:\Spin^G(V)\to\GL(W)\) be a representation. Since \(\varrho_1(-\1)=\varrho_2(-\1)\) is an involution of \(W\) (\(\varrho_1(-\1)^2=\Id_W\)), \(\varrho_1(-\1)\) has eigenvalues \(\pm 1\), and \(W=W_+\oplus W_-\), where \(W_\pm\) is the \(\pm 1\) eigenspace of \(\varrho_1(-\1)\), upon which the involution acts as \(\pm \Id_{W_\pm}\). Since \(-\1\) is central in both \(\Spin(V)\) and \(G\), it follows that \(\varrho_1\) and \(\varrho_2\) preserve \(W_\pm\), hence so does \(\varrho\). Thus if \(\varrho\) is irreducible, we have \(W=W_+\) or \(W=W_-\).
\end{proof}

We thus have a natural \(\ZZ_2\)-grading on representations of \(\Spin^G(V)\); morphisms and tensor products of representations respect parity in the obvious ways, whence representations of \(\Spin^G(V)\) form a monoidal subcategory of the category of super-vector spaces.

\subsection{Associated bundles and sections}
\label{sec:spin-G-assoc-bundles}

Now let \((M,g)\) be a pseudo-Riemannian manifold equipped with a spin-\(G\) structure \(\varpihat:\Phat\to F_{SO}\). We can also define parity for associated vector bundles of the principal bundle \(\Phat\to M\): the bundle \(\Wbundle:\Phat\times_\varrho W \to M\) is \emph{even} or \emph{odd} if the representation \(\varrho:\Spin^G(s,t)\to\GL(W)\) that induces it is. If \(\varrho:\Spin^G(s,t)\to\GL(W)\) is an even representation then we have a natural isomorphism of vector bundles
\begin{equation}
	\Wbundle = \Phat \times_\varrho W \cong \qty(F_{SO}\times_M \Qbar)\times_{\overline\varrho} W
\end{equation}
where we recall that \(\Qbar=\Phat\times_{\pi_G}\Gbar\to M\) is the canonical \(\Gbar\)-bundle, while if \(\varrho\) is odd, no such isomorphism exists. Note that if either \(G\) or \(\Spin(s,t)\) act trivially on \(W\), we have
\begin{align}
	& \Phat \times_\varrho W \cong F_{SO}\times_{\overline{\varrho}_1} W
	&& \text{or}
	&& \Phat \times_\varrho W \cong \Qbar\times_{\overline{\varrho}_2} W
\end{align}
respectively, as vector bundles. In particular, we can realise \(TM\) as an even bundle
\begin{equation}\label{eq:TM-assoc-spin-G}
	\Phat \times_{\pihat} \RR^{s,t} \cong F_{SO} \times_{\Id_{SO(s,t)}} \RR^{s,t} \cong  TM,
\end{equation}
by mapping \([p,v] \mapsto \qty[\varpihat(p)=(\varpihat(p)_i)_{i=1}^{s+t},v=(v^j)_{j=1}^{s+t}] \mapsto \sum_{i=1}^{s+t}v^i\varpihat(p)_i\) where we recall that \(\varpihat(p)=(\varpihat(p)_i)_{i=1}^{s+t}\) is nothing but an orthonormal basis for \(T_xM\) where \(p\) lies in the fibre \(\Phat_x\) above \(x\in M\). We can define similar isomorphisms for \(T^*M\), \(\Wedge^\bullet T^*M\), etc. Particularly important for connections on a spin-\(G\) structure, to be discussed soon, is the sequence of vector bundle isomorphisms
\begin{equation}
\begin{split}
	\Phat\times_{\Ad^{\Spin^G(s,t)}}\qty(\fso(s,t)\oplus\fg)
	& \cong \qty(\Phat\times_{\Ad^{\SO(s,t)}\circ\pihat} \fso(s,t))\oplus_M \qty(\Phat\times_{\Ad^{\Gbar}\circ\pi_G} \fg)	\\
	& \cong \qty(F_{SO}\times_{\Ad^{\SO(s,t)}} \fso(s,t))\oplus_M \qty(\Qbar\times_{\Ad^{\Gbar}} \fg)
\end{split}
\end{equation}
where we recall the relation between adjoint representations described in equation \eqref{eq:spin-G-adjoint}; that is, 
\begin{equation}\label{eq:adjoint-bundles-isom}
	\ad \Phat\cong \ad F_{SO} \oplus_M \ad \Qbar.
\end{equation}
Note in particular that \(\ad F_{SO}\) and \(\ad \Qbar\) can also be considered as (even) associated bundles of \(\Phat\) by restricting the adjoint representation of \(\Spin^G(s,t)\), or equivalently by composing \(\pihat\) or \(\pi_G\) with the adjoint representation of the target group thereof. 

Any representation \(\varrho:\Spin^G(s,t)\to\GL(W)\) induces, via (left) conjugation on \(\End W\), a representation \(\End\varrho:\Spin^G(s,t)\to\GL(\End W)\) for which we have
\begin{equation}
	\Phat \times_{\End\varrho}{\End W} \cong \End\Wbundle,
\end{equation}
where the latter is the bundle of fibre-wise endomorphisms of \(\Wbundle\). Differentiating the representation map \(\varrho\) gives us a \(\Spin^G(s,t)\)-equivariant Lie algebra representation \(\varrho_*:=d_\1\varrho:\fso(s,t)\oplus\fg\to\End W\) which induces a vector bundle map
\begin{equation}\label{eq:adPhat-acts-W}
	\ad\Phat\cong \ad F_{SO} \oplus_M \ad \Qbar \longrightarrow \End\Wbundle.
\end{equation}
We thus have actions of the fibres of the adjoint bundles over a point \(x\in M\) on \(\Wbundle_x\) which can be described as follows. Fixing a base-point \(p\in\Phat_x\) picks out base-points \(\varpihat(p)\in (F_{SO})_x\) and \(\varpi_G(p)\in\Qbar_x\). Then for \(A\in\fso(s,t)\), \(a\in\fg\) and \(w\in W\), using subscripts to indicate which bundle each point lies in, we have
\begin{equation}\label{eq:adPhat-W-action}
	[p,(A+a)]_{\ad\Phat}\cdot[p,w]_{\Wbundle} = [p,\varrho_*(A+a)w]_{\Wbundle},
\end{equation}
and of course \([p,(A+a)]_{\ad\Phat}\cdot[p,w]_{\Wbundle} = [\varpihat(p),A]_{\ad F_{SO}}\cdot[p,w]_{\Wbundle} + [\varpi_G(p),a]_{\ad\Qbar}\cdot[p,w]_{\Wbundle}\), where
\begin{equation}\label{eq:adFSO-adQbar-W-action}
\begin{aligned}
	[\varpihat(p),A]_{\ad F_{SO}}\cdot[p,w]_{\Wbundle} &= [p,\varrho_*(A)w]_{\Wbundle}
	&&\text{and}
	& [\varpi_G(p),a]_{\ad\Qbar}\cdot[p,w]_{\Wbundle} &= [p,\varrho_*(a)w]_{\Wbundle}.
\end{aligned}
\end{equation}
We similarly define Lie brackets on each fibre of the adjoint bundles, turning these actions into Lie algebra representations. These definitions are independent of the choice of \(p\); a different choice merely picks out a different set of isomorphisms (of Lie algebras, resp. modules of Lie algebras) of fibres \(\ad_x\Phat, \ad_x F_{SO}, \ad_x \Qbar\) and \(\Wbundle_x\) with typical fibres \(\mathfrak{spin}^G(s,t),\fso(s,t),\fg\) and \(W\).

Applying this construction to \(TM\cong\Phat\times_{\pihat}\RR^{s,t}\) gives an embedding \(\ad F_{SO}\hookrightarrow\End(TM)\) whose image over each \(x\in M\) consists of the \(g\)-skew-symmetric endomorphisms of \(T_xM\), i.e. those \(A\in\End(T_xM)\) satisfying
\begin{equation}
	g_x(X,A(Y))+g_x(A(X),Y)=0
\end{equation}
for all \(X,Y\in T_xM\). As such, we will identify \(\Gamma(\ad F_{SO})\) with \(\fso(M,g)\), the space of \(g\)-skew-symmetric sections of \(\End(TM)\). In fact, this space is an infinite-dimensional Lie algebra over \(\RR\) with the commutator bracket, corresponding to a bracket on \(\Gamma(\ad F_{SO})\) defined fibre-wise as above. Similarly, \(\Gamma(\Phat)\) and \(\fG:=\Gamma(\ad\Qbar)\) are also infinite-dimensional Lie algebras with \(\Gamma(\Phat)\cong \fso(M,g)\oplus\fG\). Finally, we have Lie algebra module actions \(\Gamma(\Phat)\otimes\Gamma(\Wbundle)\to\Gamma(\Wbundle)\), \(\fso(M,g)\otimes\Gamma(\Wbundle)\to\Gamma(\Wbundle)\) and \(\fG\otimes\Gamma(\Wbundle)\to\Gamma(\Wbundle)\), also defined fibre-wise.

\subsection{Connections and curvatures}
\label{sec:connections}

Let us now consider placing a connection on a spin-\(G\) structure \(\varpihat:\Phat\to F_{SO}\) on \((M,g)\). A similar discussion appears in \cite{Nagase1995}.

Recall that \(\Lie\Spin^G(s,t)\cong \fso(s,t)\oplus \fg\) and that its adjoint representation is given by equation \eqref{eq:spin-G-adjoint}. A (principal or Ehresmann) connection on \(\Phat\to M\) is a Lie algebra-valued 1-form \(\eA\in\Omega^1(\Phat;\fso(s,t)\oplus\fg)\) satisfying
\begin{equation}
\begin{aligned}
	R_{[a,r]}^*\eA &= \Ad^{\Spin^G(s,t)}_{[a,r]^{-1}}\circ\eA
	&& \text{and}
	& \eA(\xi_{(X,x)}) &= (X,x)
\end{aligned}
\end{equation}
for all \(a\in\Spin(s,t)\), \(r\in G\), \(X\in\fso(s,t)\) and \(x\in \fg\), where \(R_{[a,r]}:\Phat\to\Phat\) denotes the right action of \([a,g]=\pihat(a,g)\) and \(\xi_{(X,x)}\in\fX(\Phat)\) denotes the fundamental vector field generated by \((X,x)\). Letting \(\varpi_G:\Phat\to \Qbar=\Phat\times_{\pi_G}\Gbar\) denote projection to the canonical \(\Gbar\)-bundle, one can show that
\begin{equation}\label{eq:principal-connections-add}
	\eA = \varpihat^*\omega + \varpi_G^* \alpha,
\end{equation}
where \(\omega\in \Omega^1(F_{SO};\fso(s,t))\), \(\alpha\in \Omega^1(\Qbar;\fg)\) are principal connections, and conversely, any two such connections on \(F_{SO}\to M\) and \(\Qbar\to M\) induce a connection on \(\Phat\to M\). In what follows, we will demand that \(\omega\) is the Levi-Civita connection; equivalently, we can assign a torsion to \(\eA\) using the isomorphism of vector bundles in equation \eqref{eq:TM-assoc-spin-G} and demand that this torsion vanishes. This torsion is equal to that of the associated metric connection \(\omega\), so vanishes if and only if \(\omega\) is the Levi-Civita connection.

The curvature 2-forms \(\Omega_\eA\in\Omega^2(\Phat;\fso(s,t)\oplus\fg)\), \(\Omega_\omega\in\Omega^2(F_{SO};\fso(s,t))\), and \(\Omega_\alpha\in\Omega^2(\Qbar;\fg)\) satisfy
\begin{equation}\label{eq:principal-curvatures-add}
	\Omega_\eA = \varpihat^*\Omega_\omega + \varpi_G^*\Omega_\alpha,
\end{equation}
and if \(\omega\) is the Levi-Civita connection, \(\Omega_\omega\) is of course the Riemann 2-form. The curvature can of course be encoded in a different way; by standard theory of principal connections, these 2-forms are \emph{basic} (horizontal and invariant), whence they descend to sections \(\Rhat\in\Omega^2(M;\ad\Phat)\), \(R\in\Omega^2(M;\ad F_{SO})\) (the Riemann curvature), and \(F\in\Omega^2(M;\ad \Qbar)\) (the field strength of local gauge fields for \(\alpha\)) on \(M\) satisfying
\begin{equation}\label{eq:base-curvatures-add}
	\Rhat(X,Y) = R(X,Y) + F(X,Y)
\end{equation}
for all \(X,Y\in\fX(M)\), where we implicitly use the isomorphism of vector bundles in \eqref{eq:adjoint-bundles-isom}.

The principal connection \(\eA\) on \(\Phat\to M\) induces Koszul connections on associated bundles which we will denote by \(\nablahat\) and call \emph{twisted covariant derivatives}. On even vector bundles, which we have seen can be considered as associated bundles to \(F_{SO}\times_M\Qbar\), we can write \(\nablahat=\nabla + \alpha\) in local trivialisations where \(\nabla\) denotes the covariant derivative of the Levi-Civita connection and we abuse notation to write \(\alpha\) for local gauge fields of the connection \(\alpha\) on \(\Qbar\); if the action of \(R\) on the defining representation is trivial, in particular on \(TM\), we have \(\nablahat=\nabla\). More generally, it is sometimes convenient to locally abuse notation and write \(\nablahat=\nabla + \alpha\) even on non-even associated bundles, where neither \(\nabla\) nor \(\alpha\) are globally defined.

We also have the following. The first two items are properties of covariant derivatives, see \cite[Ch.III Prop.1.2]{Nomizu1963}. The others are a standard computation exercise for connections on associated bundles.

\begin{proposition}
	If \(\varpihat:\Phat\to F_{SO}\) is a spin-\(G\) structure on \((M,g)\) with torsionless principal connection \(\eA\), and \(\Wbundle=\Phat\times_\varrho W\) is an associated bundle, then the covariant derivative \(\nablahat:\fX(M)\otimes\Gamma(\Wbundle)\to\Gamma(\Wbundle)\) has the following properties:
	\begin{enumerate}\label{eq:nablahat-props}
		\item \(\nablahat_{fX+gY}\phi = f\nablahat_X\phi + g\nablahat_Y\phi\) for all \(X,Y\in\fX(M)\), \(f,g\in C^\infty(M)\) and \(\phi\in\Gamma(\Wbundle)\);
		\item \(\nablahat_X (f\phi+g\psi) = (\eL_Xf)\phi + f\nablahat_X\phi + (\eL_Xg)\psi +g\nablahat_X\psi\) for all \(X\in\fX(M)\), \(f,g\in C^\infty(M)\) and \(\phi,\psi\in\Gamma(\Wbundle)\) (where we use scalar multiplication of \(C^\infty(M)\) on \(\Gamma(\Wbundle)\));
		\item \(\nablahat_X \comm{A}{B} = \comm{\nablahat_XA}{B} + \comm{A}{\nablahat_XB}\) for all \(X\in\fX(M)\) and \(A,B\in\Gamma(\Phat)\);
		\item \(\nablahat_X (A\cdot\phi) = (\nablahat_XA)\cdot\phi + A\cdot(\nablahat_X\phi)\) for all \(X\in\fX(M)\) and \(A\in\Gamma(\Phat)\) and \(\phi\in\Gamma(\Wbundle)\).
	\end{enumerate}
	In the last two points, we have used the bracket on \(\Gamma(\Phat)\) and the action of \(\Gamma(\Phat)\) on \(\Gamma(\Wbundle)\) described in \S\ref{sec:spin-G-assoc-bundles}.
\end{proposition}

Finally, we have the following formulae for the curvatures above when acting as operators on sections of an associated bundle \(\phi\in\Gamma(\Wbundle)\):
\begin{equation}\label{eq:curvatures-derivative-def}
\begin{aligned}
	\Rhat(X,Y)\cdot\phi 
		&= \comm{\nablahat_X}{\nablahat_Y}\phi - \nablahat_{\comm{X}{Y}}\phi,	\\
	R(X,Y)\cdot\phi	
		&= \comm{\nabla_X}{\nabla_Y}\phi - \nabla_{\comm{X}{Y}}\phi,			\\
	F(X,Y)\cdot\phi	&= (\nabla_X\alpha)(Y)\cdot\phi - (\nabla_Y\alpha)(X)\cdot\phi + \comm{\alpha(X)}{\alpha(Y)}\cdot\phi. 
\end{aligned}
\end{equation}
Note that although the LHSs are defined globally, as is the RHS of the first equation, the RHS of the second equation is well-defined globally only on sections of even associated bundles or in local trivialisations, and the third is valid only in local trivialisations; we have abused notation to write \(\alpha\) for local gauge fields of the connection on \(\Qbar\).

\section{Covariant Lie derivative and Cartan calculus}
\label{sec:cov-lie-der-cartan}

\subsection{Covariant Lie derivative}
\label{sec:cov-lie-der}

Let us now introduce a notion of Lie derivatives for sections of associated bundles to a spin-\(G\) structure which generalises the Kosmann spinorial Lie derivative \cite{Kosmann1971}. Like the Kosmann derivative, it will only be defined along Killing vectors on the base and not general vector fields as for the ordinary Lie derivative.

Our starting point is the observation, due to Kostant \cite{Kostant1955}, that for any \(X\in\fX(M)\), the Lie derivative of a tensor field \(T\in\Gamma(\Otimes^pTM\otimes\Otimes T^*M)\) can be written as
\begin{equation}\label{eq:lie-der-tensor-formula}
	\eL_X T = \nabla_X T + A_X \cdot T,
\end{equation}
where \(A_X:=-\nabla X\) is the endomorphism of the tangent bundle given by \(Y\mapsto \nabla_Y X\), which acts via the tensor product of representations.

We now recall that for Killing vectors \(X\in\fiso(M,g)\), \(A_X\) lies in \(\fso(M,g)\cong\Gamma(\ad F_{SO})\), whence it naturally acts on sections of associated bundles to any spin-\(G\) structure, not just on tensor fields.

\begin{definition}[Covariant Lie derivative]\label{def:cov-lie-der}
	Let \(\varpihat:\Phat\to F_{SO}\) be a spin-\(G\) structure on an oriented pseudo-Riemannian manifold \((M,g)\), let \(\varrho:\Spin^G(s,t)\to\GL(W)\) be a representation of the spin-\(G\) group and let \(\Wbundle=\Phat\times_\varrho W \to M\) be the associated bundle. Let \(\eA\) be a torsion-free principal connection on \(\Phat\to M\) and \(\nablahat\) be the associated covariant derivative on \(\Wbundle\). Then for a Killing vector \(X\in\fiso(M,g)\), the \emph{covariant Lie derivative along \(X\) (on \(\Wbundle\))} is the operator \(\eLhat_X: \Gamma(\Wbundle)\to\Gamma(\Wbundle)\) given by
	\begin{equation}\label{eq:eLhat-nablahat}
		\eLhat_X\phi = \nablahat_X\phi + A_X\cdot\phi
	\end{equation}
	for \(\phi\in\Gamma(\Wbundle)\), where \(A_X=-\nabla X\).
\end{definition}

A similar derivative on spinor fields was introduced via local expressions in \cite[Sec.7]{deMedeiros2018} in the special case of signature \((1,5)\) with, in the terminology of the present work, a reducible spin-\(h\) (that is, spin-\(\Sp(1)\)) structure with trivial canonical \(\overline{\Sp(1)}=\SO(3)\)-bundle and trivial \(\Sp(1)\) lift thereof. The proposition below is also generalises statements in loc. cit. A more general notion of covariant Lie derivative is reviewed in \cite{Giotopoulos2025}.

The covariant Lie derivative agrees with the ordinary Lie derivative on sections even bundles for which \(G\) acts trivially on the defining representation, in particular on \(\fX(M)\) and \(\fso(M,g)\). It agrees with \(\nablahat\) on those for which \(\Spin(s,t)\) acts trivially, in particular on \(\fG\). It also satisfies the following basic properties, justifying its name.

\begin{proposition}
	The covariant Lie derivative associated to the connection \(\eA\) satisfies the following:
	\begin{enumerate}
		\item \(\eLhat_{fX+gY}\phi = f\eLhat_X\phi + g\eLhat_Y\phi\) for all \(X,Y\in\fiso(M,g)\), \(f,g\in C^\infty(M)\) and \(\phi\in\Gamma(\Wbundle)\);
		\item \(\eLhat_X (f\phi+g\psi) = (\eL_Xf)\phi + f\eLhat_X\phi + (\eL_Xg)\psi +g\eLhat_X\psi\) for all \(X\in\fiso(M,g)\), \(f,g\in C^\infty(M)\) and \(\phi,\psi\in\Gamma(\Wbundle)\) (where we use scalar multiplication of \(C^\infty(M)\) on \(\Gamma(\Wbundle)\));
		\item \(\eLhat_X \comm{A}{B} = \comm{\eLhat_XA}{B} + \comm{A}{\eLhat_XB}\) for all \(X\in\fiso(M,g)\) and \(A,B\in\Gamma(\Phat)\);
		\item \(\eLhat_X (A\cdot\phi) = (\eLhat_XA)\cdot\phi + A\cdot(\eLhat_X\phi)\) for all \(X\in\fiso(M,g)\) and \(A\in\Gamma(\Phat)\) and \(\phi\in\Gamma(\Wbundle)\).
	\end{enumerate}
\end{proposition}

\begin{proof}
	Each property follows from the analogous property for \(\nablahat\) in Proposition~\ref{eq:nablahat-props} and, for the last two, the fact that the action of \(\Gamma(\Phat)\) on \(\Gamma(\Wbundle)\) is a Lie algebra action.
\end{proof}

We note that \(\eLhat\) also respects Leibniz rules with respect to the wedge product and Clifford multiplication of forms as well as the Clifford action of differential forms on bundles of Clifford modules. We also have the following.

\begin{proposition}\label{prop:eLhat-props}
	The covariant Lie derivative also satisfies the following properties:
	\begin{align}
		\comm{\eLhat_X}{\nablahat_Y} &= \nablahat_{\comm{X}{Y}} + F(X,Y) \label{eq:eLhat-nablahat-comm}
	\end{align}
	for \(X\in\fiso(M,g)\) and \(Y\in\fX(M)\) (equivalently, \(\comm{\eLhat_X}{\nablahat} = \imath_X F\) for all \(X\in\fiso(M,g)\)) and
	\begin{equation}\label{eq:eLhat-comm}
		\comm{\eLhat_X}{\eLhat_Y} = \eLhat_{\comm{X}{Y}} + F(X,Y)	
	\end{equation}
	for \(X,Y\in\fiso(M,g)\).\footnote{
		Note that the \(F(X,Y)\) terms in these equations are to be understood as operators acting on sections of an associated bundle via the appropriate representation of \(\fg\) (using equation~equation~\eqref{eq:adFSO-adQbar-W-action} fibre-wise).
	}
\end{proposition}

\begin{proof}
	We will make use of some identities due to Kostant \cite{Kostant1955}. First, if \(X\) is a Killing vector then \(\nabla_YA_X = R(X,Y)\), so
	\begin{equation}
	\begin{split}
		\comm{\eLhat_X}{\nablahat_Y} 
			&= \comm{\nablahat_X}{\nablahat_Y} + \comm{A_X}{\nablahat_Y}	\\
			&= \comm{\nablahat_X}{\nablahat_Y} - \nabla_Y A_X	\\
			&= \nablahat_{\comm{X}{Y}} + \Rhat(X,Y) - R(X,Y)	\\
			&= \nablahat_{\comm{X}{Y}} + F(X,Y),
	\end{split}
	\end{equation}
	where we have used the curvature identities \eqref{eq:base-curvatures-add} and \eqref{eq:curvatures-derivative-def}; this demonstrates the first property. A Leibniz rule, the formula \eqref{eq:lie-der-tensor-formula} for the Lie derivative and the identities \(\nabla_X A_Y = - R(X,Y)\) and \(A_{\comm{X}{Y}} = \comm{A_X}{A_Y} - R(X,Y)\) for Killing vectors \(X,Y\) yield
	\begin{equation}
		\comm{\eLhat_X}{A_Y}=\eL_X A_Y = \nabla_X A_Y + \comm{A_X}{A_Y} = - R(X,Y) + \comm{A_X}{A_Y} = A_{\comm{X}{Y}},
	\end{equation}
	and using this as well as the first property gives us
	\begin{equation}
		\comm{\eLhat_X}{\eLhat_Y}
			= \comm{\eLhat_X}{\nablahat_Y}  + \comm{\eLhat_X}{A_Y} 
			= \nablahat_{\comm{X}{Y}} + F(X,Y) + A_{\comm{X}{Y}} 
			= \eLhat_{\comm{X}{Y}} + F(X,Y),
	\end{equation}
	which demonstrates the second property.
\end{proof}

In particular, if \(\alpha\) is not a flat connection, the covariant Lie derivative does not define a representation of the Lie algebra of Killing vectors \(\fiso(M,g)\) on sections of associated bundles. In Section~\ref{sec:symm-alg-spin-g}, we will construct an extension of \(\fiso(M,g)\) which does possess such a representation.

\subsection{Covariant Cartan calculus}
\label{sec:cov-cartan-calculus}

Recall that the choice of connection \(\eA\) on \(\Phat\) induces not just a covariant derivative on all associated bundles, but a \emph{covariant exterior derivative}, which we will denote by \(\dhat\), on differential forms with values in an associated bundle. With this notation, we have Bianchi identities \(\dhat\Rhat =0\) and \(\dhat F = 0\).

The ordinary Lie derivative, exterior derivative and interior derivative (contraction) on differential forms satisfy a number of identities collectively known as \emph{Cartan calculus} which are extremely useful for computations. We will now show that our covariant derivatives satisfy a similar \emph{covariant Cartan calculus} when acting on differential forms with values in bundles associated to the spin-\(G\) structure.

\begin{proposition}[Covariant Cartan calculus]\label{prop:cov-cartan-calculus}
	Let \(\varrho:\Spin^G(s,t)\to \GL(W)\) be a representation with associated bundle \(\Wbundle = \Phat\times_\varrho W \to M\). Then we have the following Cartan formula for the covariant Lie derivative along any Killing vector \(X\) when acting on \(\Omega^\bullet(M;\Wbundle):=\Gamma\qty(\Wedge^\bullet T^*M\otimes \Wbundle)\):
	\begin{equation}
		\eLhat_X = \imath_X\dhat + \dhat\imath_X.
	\end{equation}
	We also have the following identities of operators on \(\Omega^\bullet(M;\Wbundle)\):\footnote{
			Here and elsewhere, the wedge of a form with values in \(\ad\Qbar\) (or locally in \(\fg\)) with a form with values in \(\Wbundle\) is understood to include the action of the former on the latter via the representation of \(\fg\) on \(W\) (as in equation~\eqref{eq:adFSO-adQbar-W-action}).
			}
	\begin{equation} \label{eq:eLhat-dhat-comm}
		\comm{\eLhat_X}{\dhat} = \imath_X F \wedge,
	\end{equation}
	for all \(X\in\fiso(M,g)\);
	\begin{equation}\label{eq:eLhat-i-comm}
		\comm{\eLhat_X}{\imath_Y} = \imath_{\comm{X}{Y}},
	\end{equation}
	for all \(X\in\fiso(M,g)\) and \(Y\in\fX(M)\);
	\begin{equation}\label{eq:eLhat-comm-forms}
		\comm{\eLhat_X}{\eLhat_Y} = \eLhat_{\comm{X}{Y}} + F(X,Y),
	\end{equation}
	for all \(X,Y\in\fiso(M,g)\).
\end{proposition}

\begin{proof}
	The covariant Cartan formula can be proven in a number of ways. We will show it locally, using the Cartan formula for the ordinary Lie derivative. For \(\omega\in\Omega^k(M;\Wbundle)\), in a local trivialisation, we treat \(\omega\) as taking values in \(W\) and the local gauge 1-form \(\alpha\) as taking values in \(\fg\) and compute
	\begin{equation}
	\begin{split}
		\eLhat_X\omega &= \eL_X\omega + (\imath_X\alpha)\cdot \omega	\\
		&= \imath_Xd\omega + d\imath_X\omega + (\imath_X\alpha)\cdot \omega	\\
		&= \imath_X\dhat\omega + \dhat\imath_X\omega - \imath_X(\alpha\wedge\omega) + (\imath_X\alpha)\cdot\omega - \alpha\wedge\imath_X\omega.
	\end{split}
	\end{equation}
	The last three terms in the final line cancel by a Leibniz rule for the contraction \(\imath\). The next identity follows from the Cartan formula and the standard identity \(\dhat^2=F\wedge\) (which one can verify by locally writing \(\dhat=d+\alpha\wedge\)):
	\begin{equation}
	\begin{split}
		\comm{\eLhat_X}{\dhat}\omega
			&= \imath_X\dhat^2\omega + \dhat\imath_X\dhat\omega - \dhat\imath_X\dhat\omega - \dhat^2\imath_X\omega \\
			&= \imath_X(F\wedge \omega) - F\wedge\imath_X\omega	\\
			&= (\imath_X F) \wedge \omega,
	\end{split}
	\end{equation}
	once again using the Leibniz identity for the contraction. The next identity is nothing but a Leibniz identity for the covariant Lie derivative. We already saw the final identity as part of Proposition~\ref{prop:eLhat-props}, but we provide an alternative proof when acting on associated bundle-valued forms. Indeed, using all of the previous identities, we have
	\begin{equation}
	\begin{split}
		\comm{\eLhat_X}{\eLhat_Y}\omega
			&= \comm{\eLhat_X}{\imath_Y\dhat}\omega + \comm{\eLhat_X}{\dhat\imath_Y}\omega	\\
			&= \imath_Y\comm{\eLhat_X}{\dhat}\omega + \comm{\eLhat_X}{\imath_Y}\dhat\omega + \dhat\comm{\eLhat_X}{\imath_Y}\omega + \comm{\eLhat_X}{\dhat}\imath_Y\omega	\\
			&= \imath_Y((\imath_X F)\wedge\omega) + \imath_{\comm{X}{Y}}\dhat\omega + \dhat\imath_{\comm{X}{Y}}\omega + (\imath_X F)\wedge\imath_Y\omega	\\
			&= \eLhat_{\comm{X}{Y}}\omega + F(X,Y)\cdot \omega
	\end{split}
	\end{equation}
	for all Killing vectors \(X,Y\).
\end{proof}

\subsection{The symmetry algebra of a spin-\(G\) structure}
\label{sec:symm-alg-spin-g}

Let \((M,g)\) be a connected and oriented pseudo-Riemannian manifold of signature \((s,t)\) with spin-\(G\) structure \(\varpihat:\Phat\to F_{SO}\) with torsion-free principal connection \(\eA\).
As in \S\ref{sec:spin-G-assoc-bundles}, we identify the space of sections of \(\ad F_{SO}\) with \(\fso(M,g)\), the Lie algebra of \(g\)-skew-symmetric sections of \(\End(TM)\), and let \(\fG\) denote the Lie algebra of section of \(\ad\Qbar\) where \(\Qbar=\Phat\times_{\pi_G}\Gbar\to M\) is the canonical \(\Gbar\)-bundle.
We previously noted that, due to \eqref{eq:eLhat-comm}, the covariant Lie derivative fails to define a representation of the Lie algebra of Killing vectors \(\fiso(M,g)\) on sections of associated bundles. We will now define a Lie algebra \(\fsymm(\varpihat,\eA)\), the symmetry algebra, which does admit a representation on such sections.

We first note that since \(\ad Q\) is an associated bundle to the spin-\(G\) structure (via the adjoint action of \(\Spin^G(s,t)\)), the covariant Lie derivative defines an \(\RR\)-linear map \(\eLhat:\fiso(M,g)\otimes\fG\to\fG\); since \(\comm{\fso(s,t)}{\fg}=0\), for all \(X\in\fiso(M,g)\) and \(\gamma\in\fG\) we have \(A_X\cdot \gamma = \ad_{A_X}(\gamma) = \comm{A_X}{\gamma} = 0\) (where \(A_X:=-\nabla X\)) in \(\fiso(M,g)\oplus\fG\), so
\begin{equation}
	\eLhat_X\gamma = \nablahat_X\gamma.
\end{equation}
We also recall that the curvature \(F\) of the connection \(\alpha\) on \(\Qbar\) is an element of \(\Omega^2(M;\ad\Qbar)\), so for \(X,Y\in\fX(M)\) we have \(F(X,Y)\in\fG\). We then have the following.

\begin{proposition}\label{prop:symm-alg}
	Let \(\varpihat:\Phat\to F_{SO}\) be spin-\(G\) structure on \((M,g)\) and \(\eA\in\Omega^2(\Phat,\fso(s,t)\oplus\fg)\) a torsionless connection.
	Then the pair \(\qty(\fsymm(\varpihat,\eA) = \fiso(M,g)\oplus\fG,\ccomm{-}{-})\) is a Lie algebra, where the bracket \(\ccomm{-}{-}\) is defined by
	\begin{equation}
	\begin{aligned}
		\ccomm{X}{Y} &= \comm{X}{Y}_{\fiso(M,g)} + F(X,Y) = \eL_X Y + F(X,Y),\\
		\ccomm{X}{\gamma} &= \eLhat_X\gamma = \nablahat_X\gamma,\\
		\ccomm{\gamma}{\gamma'} &= \comm{\gamma}{\gamma'}_{\fG},
	\end{aligned}
	\end{equation}
	where \(X,Y\in \fiso(M,g)\), \(\gamma,\gamma'\in\fG\), and the subscripts on the brackets on the RHS denote which Lie algebra they are to be taken in.
	
	If \(\varrho:\Spin^G(s,t)\to\GL(W)\) is a representation and \(\Wbundle=\Phat\times_\varrho W\) the associated bundle, then \(\Gamma(\Wbundle)\) is a \(\fsymm(\varpihat,\eA)\)-module, where elements of \(\fiso(M,g)\) act via \(\eLhat\) and those of \(\fG\) act fibre-wise (via equation~\eqref{eq:adFSO-adQbar-W-action}).
\end{proposition}

\begin{proof}
	Clearly the expressions given extend to a map \(\Wedge^2\fsymm(\varpihat,\eA)\to\fsymm(\varpihat,\eA)\); we need only check the Jacobi identities. Since \(\ccomm{-}{-}\) restricts to the natural Lie bracket on \(\fG\), the Jacobi identity for three elements of \(\fG\) is automatically satisfied. By a Leibniz rule for the covariant derivative we have
	\begin{equation}
		\ccomm{X}{\ccomm{\gamma}{\gamma'}} = \nablahat_X\comm{\gamma}{\gamma'} = \comm{\nablahat_X{\gamma}}{\gamma'} + \comm{\gamma}{\nablahat_X\gamma'} = \ccomm{\ccomm{X}{\gamma}}{\gamma'} + \ccomm{\gamma}{\ccomm{X}{\gamma'}},
	\end{equation}
	for \(X\in\fiso(M,g)\) and \(\gamma,\gamma'\in\fG\), and using the identity \eqref{eq:eLhat-comm} we find
	\begin{equation}
		\ccomm{X}{\ccomm{Y}{\gamma}} - \ccomm{Y}{\ccomm{X}{\gamma}}
			= \comm{\eLhat_X}{\eLhat_Y}\gamma
			= \eLhat_{\comm{X}{Y}}\gamma + \comm{F(X,Y)}{\gamma}
			= \ccomm{\ccomm{X}{Y}}{\gamma},
	\end{equation}
	for all \(X,Y\in\fiso(M,g)\) and \(\gamma\in\fG\). It remains check the Jacobi identity for three Killing vectors; this is a little more involved than the other cases. For Finally, for \(X,Y,Z\in\fiso(M,g)\), we find that
	\begin{equation}
		\ccomm{X}{\ccomm{Y}{Z}} = \eL_X\eL_Y Z + F(X,\eL_YZ) + \eLhat_X(F(Y,Z)),
	\end{equation}
	and, noting that \(\eLhat_X(F(Y,Z))=\nablahat_X(F(Y,Z))\), upon taking cyclic permutations of the above we find that the remaining component of the Jacobi identity is equivalent to
	\begin{equation}
		\qty(\nablahat_XF)(Y,Z) + \qty(\nablahat_YF)(Z,X) + \qty(\nablahat_ZF)(X,Y) = 0
	\end{equation}
	for all Killing vector fields \(X,Y,Z\). But this equation holds identically -- indeed, it is nothing but the Bianchi identity for \(F\).

	For the second part of the proposition, equation \eqref{eq:eLhat-comm} gives us
	\begin{equation}
		\ccomm{X}{Y}\cdot\phi 
			= \eLhat_{\comm{X}{Y}}\phi + F(X,Y)\cdot\phi
			= \eLhat_X\qty(\eLhat_Y\phi) - \eLhat_Y\qty(\eLhat_X\phi)
	\end{equation}
	for all \(X,Y\in\fiso(M,g)\) and \(\phi\in\Gamma(\Wbundle)\), where \(\cdot\) denotes the action of \(\fG\) on \(\Gamma(\Wbundle)\) induced by the action of \(\fg\) on \(W\). The Leibniz rule for \(\eLhat_X\) gives 
	\begin{equation}
		\ccomm{X}{\gamma}\cdot\phi
		= \qty(\eLhat_X \gamma)\cdot\phi 
		= \eLhat_X (\gamma\cdot\phi) - \gamma\cdot\qty(\eLhat_X \phi)
	\end{equation}
	for \(X\in\fiso(M,g)\), \(\gamma\in\fG\) and \(\phi\in\Gamma(\Wbundle)\), and we also have
	\begin{equation}
		\ccomm{\gamma}{\gamma'}\cdot\phi
		= \comm{\gamma}{\gamma'}\cdot\phi
		= \gamma\cdot(\gamma'\cdot\phi) - \gamma'\cdot(\gamma\cdot\phi)
	\end{equation}
	for all \(\gamma,\gamma'\in\fG\) and \(\phi\in\Gamma(\Wbundle)\), since the second equality holds pointwise for the action of \(\fg\) on \(\Gamma(\Wbundle)\). This shows that \(\Gamma(\Wbundle)\) is a module for \(\fsymm(\varpihat,\eA)\).
\end{proof}

It is manifest in the definition of the bracket \(\ccomm{-}{-}\) that \(\fsymm(\varpihat,\eA)\) is \emph{not} the direct sum of Lie algebras \(\fiso(M,g)\oplus\fG\) in general, but rather an extension of \(\fiso(M,g)\) by \(\fG\); we have a short exact sequence of Lie algebras
\begin{equation}
\begin{tikzcd}
	0 \ar[r] & \fG \ar[r] & \fsymm(\varpihat,\eA) \ar[r] & \fiso(M,g) \ar[r] & 0
\end{tikzcd}
\end{equation}
where the map \(\fG\to\fsymm(\varpihat,\eA)\) is the inclusion and \(\fsymm(\varpihat,\eA)\to\fiso(M,g)\) is projection. Note that if \(\imath_XF=0\) for all \(X\in\fiso(M,g)\), in particular if the connection \(\alpha\) is flat, then \(\ccomm{-}{-}\) restricted to \(\fiso(M,g)\) is simply the Lie bracket of vector fields, and we have \(\fsymm(\varpihat,\eA)=\fiso(M,g)\ltimes\fG\). From another point of view, \(\fsymm(\varpihat,\eA)\) is a Lie algebra deformation of \(\fiso(M,g)\oplus\fG\).

\section{Homogeneous generalised spin structures}
\label{sec:homogeneous-gen-spin}

Generalising the classic notion of homogeneous (a.k.a. equivariant) \emph{spin} structures on homogeneous pseudo-Lorentzian manifolds \cite{Bar1992,Cahen1993,Alekseevsky2019}, in this section we will discuss homogeneous \emph{generalised} spin structures which have previously appeared in e.g. \cite{Artacho2023}. See Kobayashi--Nomizu for general background on homogeneous spaces \cite[Ch.X]{Nomizu1969} and invariant connections \cite[Ch.II]{Nomizu1963}.

\subsection{Definitions}
\label{sec:homog-gen-spin-defs}

Let \((G,K,\eta)\) be a metric Klein pair (a connected Lie group \(G\) and closed subgroup \(K\) with \(K\)-invariant (pseudo-)inner product \(\eta\) on \(\fg/\fk\)) with signature \((s,t)\), let \((M=G/K,g)\) be the associated homogeneous pseudo-Riemannian \(G\)-space (which we assume to be oriented) and let \(V=T_oM\cong \fg/\fk\) where \(o=K\in M\). By taking the push-forward \(g_*\) of the diffeomorphism assigned to each element \(g\in G\), the action of \(G\) on \(M\) lifts to an action on \(TM\); the action of the isotropy group \(K\) preserves \((V,\eta)\), giving us the \emph{linear isotropy representation} \(\varphi:K\to\SO(V)\). The action of \(G\) on \(TM\) induces in turn an action on the special orthonormal frame bundle \(F_{SO}\), where
\begin{equation}\label{eq:g-f}
	g\cdot (f_i)_{i=1}^{\dim V} = (g\cdot f_i)_{i=1}^{\dim V} = (g_*f_i)_{i=1}^{\dim V}
\end{equation}
for all \(g\in G\) and all frames \(f=(f_i)_{i=1}^{\dim V}\). Fixing a frame \(f\) over \(o=K\), we may represent any \(A\in\SO(V)\) as a matrix \(\underline{A}\) with components \(A\indices{^i_j}\) in this frame where \(Af_i=A\indices{^i_j}f_i\). This gives us a Lie group isomorphism \(\SO(V)\xrightarrow{\cong}\SO(s,t)\), where \(A\) is mapped to the transformation on \(\RR^{s,t}\) represented in the standard basis by \(\underline{A}\), which allows us to view \(F_{SO}\) as an \(\SO(V)\)-principal bundle; we then have a \(G\)-equivariant isomorphism of \(\SO(V)\)-principal bundles
\begin{equation}\label{eq:homog-frame-bundle-assoc}
	G\times_\varphi\SO(V) \xlongrightarrow{\cong} F_{SO}
\end{equation}
given by \([g,A]\mapsto g\cdot f\cdot A\) where the left action of \(G\) on the first bundle is given by \(g\cdot[g',A]=[gg',A]\). and is compatible with the right action of \(\SO(V)\).

Now let \(H\) be some other Lie group and suppose that there exists a lift of the linear isotropy representation \(\varphi:K\to\SO(V)\) along the canonical map \(\pihat:\Spin^H(V)\to\SO(V)\); that is, a Lie group morphism \(\varphihat:K\to\Spin^H(V)\) making the following diagram
\begin{equation}\label{eq:lift-to-spin-H}
\begin{tikzcd}
	&	& \Spin^H(V) \ar[d,two heads,"\pihat"]	\\
	& K	\ar[r,"\varphi"] \ar[ur,dashed,"\varphihat"]& \SO(V)
\end{tikzcd}
\end{equation}
commute. Then \(\pihat:\Spin^H(V)\to\SO(V)\) induces a \(G\)-equivariant spin-\(H\) structure on \(M\),\footnote{
	A lift \(\hat{f}\in \Phat_o\) of \(f\in (F_{SO})_o\) induces an isomorphism \(\Spin^H(V)\cong\Spin^H(s,t)\) lifting the isomorphism \(\SO(V)\cong\SO(s,t)\) induced by \(f\), making \(\Phat\to M\) into a \(\Spin^H(s,t)\)-principal bundle.	
	}
\begin{equation}
	\varpihat: \Phat := G\times_{\varphihat}\Spin^H(V)\longrightarrow F_{SO}\cong G\times_\varphi\SO(V),
\end{equation}
where the left action of \(G\) on \(\Phat\) is given by \(g\cdot[g',A]=[gg',A]\) for \(g,g'\in G\), \(A\in\Spin^H(V)\) and is compatible with the right action of \(\Spin^H(V)\); clearly this lifts the action on \(F_{SO}\). We call this the \emph{homogeneous spin-\(H\) structure associated to the lift \(\varphihat:K\to\Spin^H(V)\)}. We call the manifold \(M\) together with this spin structure the \emph{homogeneous spin-\(H\) manifold} associated to the metric Klein pair \((G,K,\eta)\) and the lift \(\varphihat\).

A classic result says that for \(G\) simply connected and \(\eta\) positive-definite, lifts of the isotropy representation along \(\pi:\Spin(V)\to\SO(V)\) classify \emph{spin} structures on a homogeneous \(G\)-space \cite{Bar1992,Cahen1993}. This result fails without both hypotheses because there are obstructions to lifting the action of \(G\) from \(F_{SO}\) to a spin cover in general; obstructions occur indefinite signature because the spin group is disconnected, and in the time-orientable where the structure groups can be reduced to their connected components (i.e. reducing to \emph{strong} spin structures in the sense of \cite{Cortes2021,Shahbazi2024_1,Shahbazi2024_2}), the connected component is not simply connected for \((p\geq 2,q\geq 2)\). However, if \(G\) is simply connected and  \(M\) is a \emph{reductive} \(G\)-space, the above still holds in arbitrary signature \cite{Alekseevsky2019,Cahen1993}. Since a spin-\(H\) structure is not even a cover of the frame bundle, one cannot expect an analogous result even with the hypotheses. We therefore work in a more restrictive context, essentially assuming the existence of a lift of the action.

\begin{definition}
	A \emph{homogeneous spin-\(H\) structure} on a homogeneous pseudo-Riemannian \(G\)-space \((M,g)\) is a spin-\(H\) structure \(\varpihat:\Phat\to F_{SO}\) on \(M\) equipped with a left action of \(G\) on \(\Phat\), compatible with the right action of \(\Spin^H(s,t)\), such that \(\varpihat\) is \(G\)-equivariant.
	
	Two such structures \(\varpihat:\Phat\to F_{SO}\) and \(\varpihat':\Phat'\to F_{SO}\) are \emph{equivalent} if there exists a \(G\)-equivariant isomorphism of \(\Spin^H(V)\)-principal bundles \(\Phi:\Phat\to\Phat'\) such that the following diagram commutes:
	\begin{equation}
		\begin{tikzcd}
			& \Phat\ar[rr,"\Phi"]\ar[dr,"\varpihat"']	&	& \Phat' \ar[dl,"\varpihat'"]\\
			&	& F_{SO} &
		\end{tikzcd}.
	\end{equation}
\end{definition}

We note that a bundle map \(\Phi:\Phat\to\Phat'\) is an equivalence of homogeneous spin-\(H\) structures, if and only if it is equivariant under the left \(G\)-action \emph{and} the right \(\Spin^H(V)\)-action, and that every morphism (equivariant bundle map) of principal bundles for the same group is an isomorphism.

\subsection{Classification}
\label{sec:homog-gen-spin-class}

For any \(h\in H\), let \(\gamma_h:\Spin^H(V)\to\Spin^H(V)\) denote right-conjugation by \([\1,h]\in\Spin^H(V)\); \(\gamma_h([a,h'])=[a,h^{-1}h'h]\). If \(\varphihat:K\to\Spin^H(V)\) is a lift of the linear isotropy representation \(\varphi:K\to\SO(V)\) as in the diagram \eqref{eq:lift-to-spin-H}, then \(\gamma_h\circ\varphihat\) is also a lift for any \(h\in H\). We say that two lifts \(\varphihat,\varphihat'\) are \emph{\(H\)-conjugate} if there exits \(h\in H\) such that \(\varphihat'(k)=\gamma_h\circ\varphihat(k)\). We then have the following. See \cite{Bar1992,Cahen1993,Serrano2022,Artacho2023} for related results.

\begin{proposition}\label{prop:homogeneous-spin-H}
	Let \((G,K,\eta)\) be a metric Klein pair where \(G\) is a connected and simply connected Lie group. Then there is a one-to-one correspondence between equivalence classes of homogeneous spin-\(H\) structures on \((M=G/K,g)\) and \(H\)-conjugacy classes of lifts of the isotropy representation to \(\Spin^H(V)\) as in the diagram \eqref{eq:lift-to-spin-H}.
\end{proposition}

\begin{proof}
	Throughout this proof, fix some \(f\in(F_{SO})_o\) and use this frame, as in \eqref{eq:homog-frame-bundle-assoc} and the preceding discussion, to view \(F_{SO}\) as an \(\SO(V)\)-principal bundle and identify it with \(G\times_\varphi\SO(V)\).

	We have already shown that a lift \(\varphihat\) of the isotropy representation \(\varphi\) gives rise to a homogeneous spin-\(H\) structure. Let \(h\in H\) so that \(\varphihat':=\gamma_h\circ\varphihat\) is a conjugate lift. We must show that there is an equivalence of the associated homogeneous spin-\(H\) structures
	\begin{equation}
		\Phi: G\times_{\varphihat} \Spin^H(V) \xlongrightarrow{\cong} G\times_{\varphihat'} \Spin^H(V).
	\end{equation}
	We define \(\Phi([g,A])=[g,h^{-1}A]'\), where \(g\in G\), \(A\in\Spin^H(V)\) and the prime on the right-hand side indicates that the equivalence class is taken in \(G\times_{\varphihat'} \Spin^H(V)\). This is well-defined since if we have \(\tilde{g}\in G\) and \(\tilde{A}\in\Spin^H(V)\) such that \([\tilde{g},\tilde{A}]=[g,A]\), there exists \(k\in K\) such that \(\tilde{g}=gk^{-1}\), \(\tilde{A}=\varphihat(k)A\), so
	\begin{equation}
		[\tilde{g},h^{-1}\tilde{A}]' = [gk^{-1},h^{-1}\varphihat(k)A]' 
			= [gk^{-1},\gamma_h(\varphihat(k))h^{-1}A] 
			= [gk^{-1},\varphihat'(k)h^{-1}A]'
			= [g,h^{-1}A]'.
	\end{equation}
	It is clear from the definition that \(\Phi\) is a \(G\)-equivariant morphism of \(\Spin^H(V)\)-principal bundles covering \(F_{SO}\cong G\times_\varphi\SO(V)\), whence it is an isomorphism.
		
	On the other hand, given a homogeneous spin-\(H\) structure \(\varpihat:\Phat\to F_{SO}\), we must produce a conjugacy class of lifts \(\varphihat\) of the isotropy representation. Let us first fix a lift \(\hat{f}\in\Phat_o\) of the frame \(f\), i.e. \(\varpihat(\hat{f})=f\). This lift is unique only up to (right) multiplication by elements of \(H\); we will consider the effect of this shortly. Since the subgroup \(K\) fixes \(o\), it preserves the fibre \(\Phat_o\), upon which \(\Spin^H(V)\) acts freely and transitively, so for each \(k\in K\) there exists a unique \(\varphihat(k)\in\Spin^H(V)\) such that \(k\cdot \hat{f} = \hat{f}\cdot\varphihat(k)\). This defines a homomorphism \(\varphihat:K\to\Spin^H(V)\) since for all \(k,k'\in K\),
	\begin{equation}
		k\cdot(k'\cdot \hat{f}) = k\cdot(\hat{f}\cdot\varphihat(k')) = (k\cdot \hat{f})\cdot\varphihat(k') = (\hat{f}\cdot\varphihat(k))\cdot\varphihat(k') = \hat{f}\cdot(\varphihat(k)\varphihat(k'))
	\end{equation}
	while \((kk')\cdot \hat{f}=\hat{f}\cdot\varphihat(kk')\), so \(\varphihat(kk')=\varphihat(k)\varphihat(k')\). We can also deduce that
	\begin{equation}\label{eq:kf}
		k\cdot f = k\cdot\varpihat(\hat{f}) = \varpihat(k\cdot\hat{f}) = \varpihat(\hat{f}\cdot\varphihat(k)) = \varpihat(\hat{f})\cdot \pihat(\varphihat(k)) = f\cdot(\pihat\circ\varphihat(k)).
	\end{equation}	
	Now let us represent \(\varphi(k)\in\SO(V)\) as a matrix \([\varphi(k)\indices{^i_j}]\) in the frame \(f=(f_i)_{i=1}^{\dim V}\); this allows us to deduce that
	\begin{equation}\label{eq:kf-linear-iso}
		(f\cdot \varphi(k))_i  = \sum_{j=1}^{\dim V}f_j \varphi(k)\indices{^j_i} = \varphi(k)f_i = k\cdot f_i,
	\end{equation}
	where we use the right action of \(\SO(V)\) on \((F_{SO})_o\), the definition of the matrix elements, the natural action of \(\SO(V)\) on \(V\) and the action of \(K\subseteq G\) on \(V\). But then using \eqref{eq:kf-linear-iso}, the definition of the action of \(G\) on \(F_{SO}\) (eq.~\eqref{eq:g-f}) and \eqref{eq:kf} in turn, we have
	\begin{equation}
		f\cdot \varphi(k) = (k\cdot f_i)_{i=1}^{\dim V} = k\cdot(f_i)_{i=1}^{\dim V} = k\cdot f = f\cdot(\pihat\circ\varphihat(k)),
	\end{equation}
	whence \(\varphi(k)=\pihat\circ\varphihat(k)\) for all \(k\in K\); \(\varphihat\) is a lift of \(\varphi\). Since the definition of \(\varphihat\) involved a choice of lift \(\hat{f}\) of \(f\), we must consider the effect of choosing a different lift \(\hat{f}'\) of \(f\) and defining a lift \(\varphihat':K\to\Spin^H(V)\) of \(\varphi\) by \(k\cdot\hat{f}'=\hat{f}'\cdot\varphihat'(k)\). There exists a unique \(h\in H\) such that \(\hat{f}'=\hat{f}\cdot h\), so
	\begin{equation}
		\hat{f}'\cdot\varphihat'(k) = k\cdot (\hat{f}\cdot h) = (k\cdot\hat{f})\cdot h = (\hat{f}\cdot\varphihat(k))\cdot h = \hat{f}'\cdot(h^{-1}\varphihat(k)h) 
	\end{equation}
	thus \(\varphihat'=\gamma_h\cdot\varphihat\). Hence we have a well-defined assignment of \(H\)-conjugacy classes of lifts of the isotropy representation to homogeneous spin-\(H\) structures. We must still show that \emph{equivalent} homogeneous spin-\(H\) structures give rise to the same class of lifts. Suppose we have an equivalence \(\Phi:\Phat\to\Phat'\), define \(\hat{f}\) and \(\varphihat\) as before, define \(\hat{f}':=\Phi(\hat{f})\in \Phat'_o\) (clearly this is a lift of \(f\)), and let \(k\cdot\hat{f}'=\hat{f}'\cdot\varphihat'(k)\) define a new lift \(\varphihat':K\to\Spin^H(V)\). Then 
	\begin{equation}
		\hat{f}'\cdot\varphihat'(k) = k\cdot\Phi(\hat{f}) = \Phi(k\cdot\hat{f}) = \Phi(\hat{f}\cdot\varphihat(k)) = \hat{f}'\cdot\varphihat(k),
	\end{equation}
	whence \(\varphihat'=\varphihat\).
		
	It remains to show that the association of a homogeneous spin-\(H\) structure to a lift of the isotropy representation (up to the respective equivalences) and vice-versa are inverse to each other. In one direction, suppose we have a homogeneous spin-\(H\) structure \(\Phat\) and define \(\hat{f}\) and \(\varphihat\) as above. Let us define an equivalence of spin-\(H\) structures
	\begin{equation}
		\Phi: G\times_{\varphihat}\Spin^H(V) \xlongrightarrow{\cong} \Phat
	\end{equation}
	by \([g,A]\mapsto g\cdot\hat{f}\cdot A\). This is well-defined since for all \(k\in K\), \([gk,\varphihat(k^{-1})A]\) is mapped to
	\begin{equation}
		(gk)\cdot\hat{f}\cdot(\varphihat(k^{-1})A) = g\cdot(k\cdot\hat{f}\cdot\varphihat^{-1}(k))\cdot A = g\cdot((\hat{f}\cdot\varphihat(k))\cdot\varphihat^{-1}(k))\cdot A = g\cdot \hat{f}\cdot A,
	\end{equation}
	where we have used the definition of \(\varphihat\); it is clear from the definition that \(\Phi\) is bi-equivariant and covers \(F_{SO}\), whence it is indeed an equivalence. In the other direction, suppose we have a lift \(\varphihat\) and use it to define \(\varpihat:\Phat:=G\times_{\varphihat}\Spin^H(V)\to F_{SO}\). Recall that the isomorphism \(G\times_\varphi\SO(V)\xrightarrow{\cong} F_{SO}\) is given by \([g,A]\mapsto g\cdot f\cdot A\). If \(\hat{f}\in\Phat_o\) is a lift of \(f\) then \(\hat{f}=[1_G,h]\) for some \(h\in H\), so for all \(k\in K\) we have
	\begin{equation}
		k\cdot\hat{f} = k\cdot[1_G,h] = [k,h] = [1_G,\varphihat(k)h] = [1_G,h]\cdot(h^{-1}\varphihat(k)h)  = \hat{f}\cdot \gamma_h(\varphihat(k))
	\end{equation}
	so the lift associated to \(\varpihat:\Phat\to F_{SO}\) by the procedure described above is \(\gamma_h\circ\varphihat\), a conjugate of \(\varphihat\),  which completes the proof.
\end{proof}

\subsubsection{Reducibility}
\label{sec:homog-gen-spin-reducibility}

Let us conclude our characterisation of homogeneous spin-\(H\) structures by considering equivariant reductions. If we have a homogeneous spin-\(H\) structure \(\varpihat:\Phat\to F_{SO}\) corresponding to a lift \(\varphitilde:K\to\Spin(V)\), and there is a further lift
\(\varphitilde\times\phi:K\to\Spin(V)\times H\) along \(\nu:\Spin(V)\times H\to\Spin^H(V)\) making the following diagram commute:
\begin{equation}\label{eq:iso-rep-double-lift}
\begin{tikzcd}
	&	& \Spin(V)\times H \ar[d,two heads,"\nu"]	\\
	& K	\ar[dr,"\varphi"]\ar[r,"\varphihat"]\ar[ur,dashed,"\varphitilde\times\phi"] \ar[ur,dashed]& \Spin^H(V) \ar[d,two heads,"\pihat"]\\
	&	& \SO(V)
\end{tikzcd},
\end{equation}
then \(\varpihat\) is reducible (in the sense of Definition~\ref{def:spin-G-reducible}) since we have a homogeneous \emph{spin} structure
\begin{equation}
	P := G\times_{\varphitilde}\Spin(V) \mapsto G\times_{\varphi}\SO(V) \cong F_{SO}
\end{equation}
and an \(H\)-principal bundle \(Q := G\times_{\phi}H \to M\) such that \(P\times_M Q\) is a \(G\)-equivariant reduction of \(\Phat\):
\begin{equation}
	(P\times_M Q)\times_\nu \Spin^H(V)
		\cong (G\times_{\varphitilde\times\phi}(\Spin(V)\times H))\times_\nu \Spin^H(V)
		\cong G\times_{\phihat}\Spin^H(V)
		= \Phat.
\end{equation}

In the special case that \(\phi:K\to H\) happens to be trivial, \(Q\to M\) is the trivial \(H\)-bundle and the image of \(\varphihat\) is contained in the image of \(\Spin(V)\) in \(\Spin^H(V)\). But then since the image of \(H\) commutes with that of \(\Spin(V)\), we have \(\gamma_h\circ\varphihat=\varphihat\) for all \(h\in H\); \(\varphihat\) is \(H\)-conjugacy invariant. On the other hand, a lift \(\varphihat\) which is fixed by \(H\)-conjugation can take values in \(\Spin(V)\times_{\ZZ_2}Z(H)\subseteq\Spin^H(V)\) (where \(Z(H)\) is the centre of \(H\)) which is the image of \(\Spin(V)\) if and only if \(Z(H)=\ZZ_2=\{\pm \1\}\); if this is the case, we can trivially lift \(\varphihat\) along \(\nu\), and we conversely have that \(H\)-conjugacy invariant lifts \(\varphihat\) correspond to reducible spin-\(H\) structures with trivial \(Q\). We have \(Z(H)=\ZZ_2\) if, for example, \(H=\SO(p,q)\) or \(\SU(2)\) (i.e. in the spin-\(h\) case), but not if \(H=\SU(n)\) for \(n>2\) or \(U(n)\) for \(n\geq 1\).

\subsection{Associated bundles and connections}
\label{sec:homog-gen-spin-assoc-bundles}

The canonical \(\Hbar:=H/\ZZ_2\)-bundle can be viewed as an associated bundle to \(G\to M\) by composing the lift with the canonical map \(\pi_H:\Spin^H(V)\to\Hbar\):
\begin{equation}
	\Qbar = \Phat\times_{\pi_H}\Hbar \cong G\times_{\pi_H\circ\phihat}\Hbar;
\end{equation}
and similarly for its adjoint bundle,
\begin{equation}
	\ad\Qbar = \Phat\times_{\Ad^{\Hbar}\circ\pi_H}\fh 
		\cong G\times_{\Ad^{\Hbar}\circ\pi_H\circ\phihat}\fh.
\end{equation}

We also can adapt the discussion of associated vector bundles in \S\ref{sec:spin-G-assoc-bundles} to the homogeneous setting as follows. For any representation \(\varrho:\Spin^H(V)\to\GL(W)\), we have a natural \(G\)-equivariant vector bundle isomorphism
\begin{equation}
	\Phat\times_\varrho W = (G\times_{\varphihat}\Spin^H(V)) \times_\varrho W \xlongrightarrow{\cong} G\times_{\varrho\circ\varphihat} W
\end{equation}
given by \([g,A,w]\mapsto [g,\varrho(A)w]\). In particular,
\begin{equation}
	TM \cong \Phat\times_{\pihat} V \cong G\times_\varphi V,
\end{equation}
and similarly for other even associated bundles (\(T^*M\), \(\Wedge^pTM\), \(\ad F_{SO}\), \(\ad\Qbar\) etc.).

Recall from \S\ref{sec:connections} that on generalised spin manifolds, we do not have a canonical lift of the Levi-Civita principal connection \(\omega\) on the frame bundle \(F_{SO}\to M\) along \(\varpihat:\Phat\to F_{SO}\) (as we do on a spin structure \(\varpi:P\to F_{SO}\)) but must instead employ a connection \(\eA\) which lifts the sum of \(\omega\) and a principal connection \(\alpha\) on the canonical \(\Hbar\)-bundle \(\Qbar\to M\) along the two-sheeted cover \(\varpihat\times\varpi_H:\Phat\to F_{SO}\times_M\Qbar\); there is no canonical choice for \(\alpha\) in general.

The following lemma gives a description of these connections in the homogeneous spin-\(H\) context. It is a straightforward application of Wang's Theorem \cite{Wang1958}\cite[Ch.II,Thm.11.5]{Nomizu1969}.

\begin{lemma}[Wang's Theorem for homogeneous spin-H structures]\label{lemma:Wang-spin-H}
	Let \((G,K,\eta)\) be a metric Klein pair with homogeneous spin-\(H\) structure \(\varpihat:\Phat=G\times_{\varphihat}\Spin^H(V)\to F_{SO}\) and canonical \(\Hbar\)-bundle \(\Qbar=G\times_{\pi_H\circ\phihat}\Hbar\to M=G/K\) induced by a lift \(\varphihat:K\to\Spin^H(V)\) of the isotropy representation \(\varphi:K\to\SO(V)\), where \(V=\fg/\fk\). Then there is a one-to-one correspondence between \(G\)-invariant principal connections \(\eA\in\Omega^1(P;\fso(V)\oplus\fh)\) on \(\Phat\to M\) and linear maps \(\Phihat:\fg\to\fso(V)\oplus\fh\) such that
	\begin{enumerate}
		\item \(\Phihat\circ\Ad^G_k = \Ad^{\Spin^H(V)}_{\varphihat(k)}\circ\Phihat\) for all \(k\in K\),
		\item \(\Phihat(X)=(d_\1\varphihat)(X)\) for all \(X\in\fk\);
	\end{enumerate}
	the maps \(\Phihat\) are known as \emph{Nomizu maps}. The correspondence is given by
	\begin{equation}
		\Phihat(X) = \eA_p(\widehat\xi_X)
	\end{equation}
	where \(X\in \fg\) and \(\widehat\xi_X\) is the fundamental vector field on \(\Phat\) associated to \(X\) and \(p=[1_G,\1]\) is the natural base-point in \(\Phat_o\).
	
	Moreover, the compositions \(\pr_{\fso(V)}\circ\Phihat\) and \(\pr_{\fh}\circ\Phihat\) are Nomizu maps for the connections \(\omega\in\Omega^1(F_{SO};\fso(V))\) and \(\alpha\in\Omega^1(\Qbar;\fh)\) appearing in \eqref{eq:principal-connections-add}. The connection \(\eA\) is torsionless, or equivalently \(\omega\) is the Levi-Civita connection, if and only if
	\begin{equation}
			\pr_{\fso(V)}(\Phihat(X))\cdot\overline{Y} - \pr_{\fso(V)}(\Phihat(Y))\cdot\overline{X} - \overline{\comm{X}{Y}} = 0
	\end{equation}
	for all \(X,Y\in\fg\), where the \(X\mapsto\overline{X}\) denotes the natural map \(\fg\to\fg/\fk \cong V\).
	
	Finally, the curvature \(\Omega_\eA\in\Omega_G^2(\Phat;\fso(V)\oplus\fh)\) is given by
	\begin{equation}
		(\Omega_\eA)_p(\widehat\xi_X,\widehat\xi_Y) = \comm{\Phihat(X)}{\Phihat(Y)}_{\fso(V)\oplus\fh} - \Psi(\comm{X}{Y}_\fg),
	\end{equation}
	and by \eqref{eq:principal-curvatures-add}, the \(\fso(V)\)- and \(\fh\)-components of this expression give the curvature values \((\Omega_\omega)_p(\widehat\xi_X,\widehat\xi_Y)\) and \((\Omega_\alpha)_p(\widehat\xi_X,\widehat\xi_Y)\) respectively. These values determine the curvature 2-forms uniquely by homogeneity.
\end{lemma}

\printbibliography[heading=bibintoc,title={Bibliography}]

\end{document}